\documentclass[11pt,twoside,onecolumn,reqno]{amsart}
\usepackage[english]{babel}
\usepackage{amsmath,amsfonts,amssymb,amsthm}
\usepackage{mathabx}
\usepackage[
skip=2pt, 
labelfont=sf,
hypcap=false,
format=hang,
labelsep=period,     
textformat=period    
]{caption}
\usepackage[dvips]{graphicx}
\graphicspath{{images/}}
\usepackage[dvipsnames]{xcolor}
\usepackage{comment}

\usepackage{hyperref}
\hypersetup{
bookmarksnumbered=,
pdffitwindow=true,
pdfstartview={FitH},
pdfauthor={J. Gimeno, J. Yang, R. de la Llave},
pdftitle={State-dependent delay: computation},
pdfsubject={},
pdfcreator={},
}

\def\real{\mathbb{R}}
\def\torus{\mathbb{T}}

\title[State-dependent delay: computation]{Numerical computation of
  periodic orbits and isochrons for state-dependent delay perturbation
  of an ODE in the plane}

\author[J. Gimeno]{Joan Gimeno} 

\address{Department of Mathematics, University of Rome Tor Vergata,
  Via della Ricerca Scientificia 1, 00133 Rome (Italy)}
\email{joan@maia.ub.es}

\author[J. Yang]{Jiaqi Yang} 

\address{School of Mathematics, Georgia Institute of Technology, 686
  Cherry St. Atlanta GA. 30332-0160} \email{jyang373@gatech.edu}

\author[R. de la Llave]{Rafael de la Llave} 

\address{School of Mathematics, Georgia Institute of Technology, 686
  Cherry St. Atlanta GA. 30332-0160}
\email{rafael.delallave@math.gatech.edu}


\date{\today}

\thanks{ J.~Y. and R.~L. were partially supported by NSF grant
  DMS-1800241.  J.~G. acknowledges financial support from Spanish
  grants MDM-2014-0445, PGC2018-100699-B-I00 (MCIU/AEI/FEDER, UE),
  Catalan grant 2017 SGR 1374 and Italian grant MIUR-PRIN 20178CJA2B
  ``New Frontiers of Celestial Mechanics: theory and Applications''.
  This research was also funded by H2020-MCA-RISE \#734577 which
  supported visits of J.~G. to Georgia Inst. of Technology and of
  J.~Y. to Univ. of Barcelona.  J.~G. thanks School of Mathematics GT
  for hospitality Springs 2018 and 2019 and Fall 2019. }

\newtheorem{thm}{Theorem}[section]
\newtheorem{meta-thm}[thm]{Meta-Theorem}

\newtheorem{remark}[thm]{Remark}

\newtheorem{pro}[thm]{Proposition}
\newtheorem{lem}[thm]{Lemma}
\newtheorem{alg}[thm]{Algorithm}

\newcommand{\ve}[1]{\boldsymbol{#1}}
\newcommand{\I}{\ve i}
\let\Re\relax\DeclareMathOperator{\Re}{Re}
\let\Im\relax\DeclareMathOperator{\Im}{Im}

\begin{document}

\begin{abstract}

We present algorithms and their implementation to compute limit cycles
and their isochrons for state-dependent delay equations (SDDE's) which
are perturbed from a planar differential equation with a limit cycle.

Note that the space of solutions of an SDDE is infinite dimensional.
We compute a two parameter family of solutions of the SDDE which
converge to the solutions of the ODE as the perturbation goes to zero
in a neighborhood of the limit cycle.

The method we use formulates functional equations among periodic
functions (or functions converging exponentially to periodic).  The
functional equations express that the functions solve the SDDE.
Therefore, rather than evolving initial data and finding solutions of
a certain shape, we consider spaces of functions with the desired
shape and require that they are solutions.

The mathematical theory of these invariance equations is developed in
a companion paper, which develops \emph{a posteriori} theorems.  They
show that, if there is a sufficiently approximate solution (with
respect to some explicit condition numbers), then there is a true
solution close to the approximate one.  Since the numerical methods
produce an approximate solution, and provide estimates of the
condition numbers, we can make sure that the numerical solutions we
consider approximate true solutions.

In this paper, we choose a systematic way to approximate functions by
a finite set of numbers (Taylor-Fourier series) and develop a toolkit
of algorithms that implement the operators -- notably composition --
that enter into the theory.  We also present several implementation
results and present the results of running the algorithms and their
implementation in some representative cases.
\end{abstract}

\maketitle
\clearpage
\section{Introduction}
\label{sec.intro}

Many phenomena in nature and technology are described by limit cycles
and by now there is an extensive mathematical theory of them
\cite{Minorsky,AndronovVK87}.

These limit cycles often arise in feedback loops between effects that
pump and remove energy in ways that depend on the state of the system.
When the feedback happens instantaneously, these phenomena are modeled
by an ordinary differential equation (ODE). Nevertheless, in many real
phenomena, the feedback takes time to start acting.  In such cases,
the appropriate models are delay differential equations (DDE's) in
which the time derivative of a state is given by an expression which
involves the state at a previous time.  Such delays are documented to
be important in several areas of science and technology (e.g. in
electrodynamics, population dynamics, neuroscience, circuits,
manufacturing, etc. See \cite{Krisztin2006} for a relatively recent
survey documenting many areas where DDE's are important models).

Note that, from the mathematical point of view, adding even a small
delay term in the ODE model is a very singular perturbation since the
nature of the problem changes drastically. Notably, the natural phase
spaces in delay equations are infinite dimensional (there is some
discussions about what are the most natural ones) rather than the
finite dimensional phase spaces of ODE.

One would heuristically expect that, if the delay term is a small
quantity, there are solutions of the delay problem that resemble the
solutions of the unperturbed ODE.  Due to the singular perturbation
nature of the problem, justifying this intuitive idea is a nontrivial
mathematical task. Of course, besides the finite dimensional set of
solutions that resemble the solutions of the ODE, one expects many
other solutions, which may be very different.  \cite{Krisztin2006}.

The recent rigorous paper \cite{Jiaqi2020}, describes a formalism to
study the effect of introducing a delay to an equation in the plane
with a limit cycle.  The paper \cite{Jiaqi2020} shows that, in some
appropriate sense, the solutions of the ordinary differential equation
persist.  The method in \cite{Jiaqi2020} is constructive since it is
based on showing that the iterations of an explicit operator converge.

The goal of this paper is to present algorithms and implementation
details for the mathematical arguments developed in \cite{Jiaqi2020}.
One also expects that solutions we compute -- and which resemble the
solutions of the ODE -- capture the full dynamics of the SDDE in the
sense that the solutions of the SDDE in a neighborhood converge to
this finite dimensional solution family very fast.

The algorithms consist in specifying discretizations for all the
functional analysis steps in \cite{Jiaqi2020}.  We do not present
rigorous estimates on the effects of discretizations (they are in
principle applications of standard estimates), but we present analysis
of running times.  We have implemented the algorithms above and report
the results of running them in some representative examples.  In our
examples, one can indeed obtain very accurate solutions in a few
minutes in a standard today's laptop.  Thanks to the \emph{a
  posteriori} theorems in \cite{Jiaqi2020}, we can guarantee that
these solutions correspond to true solutions of the problem.

We recall that the method of \cite{Jiaqi2020} consists in bypassing
the evolution and formulating the existence of periodic orbits (and
the solutions converging to them) as the solutions in a class of
functions with periodicity. Furthermore, the paper \cite{Jiaqi2020}
establishes an \emph{a posteriori} theorem which states that given a
sufficiently approximate solution of the invariance equation, there is
a true solution close to it. To be more precise, an approximate
solution is sufficiently approximate, if the error is smaller than an
explicit expression involving several properties of the approximate
solution (commonly called condition numbers).

The numerical methods developed and run here, produce an approximate
solution and obtain estimates on the condition numbers. So, we can be
quite confident that the solutions produced by our numerical methods
correspond to true solutions.

\begin{remark}
 Although the results in \cite{Jiaqi2020} have been detailed for the
 case of SDDE's, they also apply without major modifications to
 advanced or even mixed differential equations.
\end{remark}

\begin{remark} 
The paper \cite{Jiaqi2020} includes a theory for different
regularities of the differential equation, but in this numerical
paper, we will only formulate results for analytic systems. Since we
will specify which derivatives appear in the calculations, it is clear
that the algorithms apply also for problems with finite
differentiability.
\end{remark}

\begin{remark} 
One of the consequences of the approach in \cite{Jiaqi2020} is that
one can easily obtain smooth dependence on parameters for the
solutions.  Note that if one studies the periodic orbits as fixed
points of an evolution operator, one needs to study the smooth
dependence of the solutions on the initial data and on parameters,
which is a delicate question for general solutions.  See \cite[Chapter
  3.7]{HaleVL}.
\end{remark}

\begin{remark}
The \emph{a posteriori} results justify that the approximate solution
is independent of the method for which it has been produced.

Besides the numerical approximations, it is also customary in applied
mathematics to produce approximate solutions using formal asymptotic
expansions.  For the problem at hand, the paper \cite{CasalCL2019}
develops formal asymptotic expansions of the periodic solutions in
powers of the term in the delay.

The expansions in \cite{CasalCL2019} are readily computable with the
methods presented here. They can be taken as starting points for the
fixed point method in \cite{Jiaqi2020}. Moreover, the \emph{a
  posteriori} results of \cite{Jiaqi2020} show that these expansions
are asymptotic in a very strong sense.
\end{remark}

\begin{remark}
The paper \cite{Jiaqi2020} also includes some local uniqueness
statements of the solutions (under the condition that the solutions
have a certain shape). Hence the numerically computed approximate
solutions of the invariance equation identify a unique solution, which
is unambiguous. This uniqueness is crucial to compare different
numerical runs as well as to obtain smooth dependence on parameters.

The uniqueness in \cite{Jiaqi2020} is somewhat subtle.  The limit
cycle is unique as well as the Taylor expansions of isochrons (their
parameterizations are unique once we fix origins of coordinates and
scales).  On the other hand, the full isochrons are unique only when
one specifies a cut-off. Similar effects happen in the study of center
manifolds \cite{Sijbrand85}.

From the numerical point of view, we only compute the limit cycle and
a finite Taylor expansion of the isochrons. The error of the reminder
of the Taylor expansion is indeed very small (much smaller than other
sources of numerical error, which are already small).
\end{remark}

\subsection{Organization of the paper}
The paper is organized in an increasing level of details trying to
guide the reader from the general steps of the algorithms to the more
specialized and hardest steps of them.

First of all, we detail in section~\S\ref{sec.overview} an overview of
the method developed in \cite{Jiaqi2020}. In particular, we first
introduce the situation of the unperturbed problem in
\S\ref{sec.overview-unperturbed} in order to move to the perturbed
problem in \S\ref{sec.overview-perturbed}.  This will lead to the
explicit expression of the invariance equation in
\S\ref{sec.overview-inv} and the periodicity and normalization
conditions in \S\ref{sec.overview-pc} and \S\ref{sec.overview-nc}.

Our results start from the unperturbed case in \cite{Huguet2013}, we
will summarize in \S\ref{sec.unperturbed-case} the steps and add
practical comments for numerically computing a parameterization in the
unperturbed case.

The algorithms that allow to solve the invariance equation introduced
in \S\ref{sec.overview-inv} are fully detailed in section
\S\ref{sec.perturbed-case}.

The numerical composition of periodic mappings as well as its
computational complexity needs special care. Hence,
section~\S\ref{sec.composing-periodic-maps} explains in detail such a
process in a Fourier representation.

Finally, section~\S\ref{sec.experiment} reports the results of some
numerical experiments.

\section{Overview of the problem and the method}
\label{sec.overview}
\subsection{The parameterization method for limit cycles
  and their isochrons in ODE's}
\label{sec.overview-unperturbed}

Our starting point is the main result in \cite{Huguet2013}, which we
recall informally (omitting precisions on regularity, domains of
definition, etc).

Given an analytic ordinary differential equation (ODE) in the plane
\begin{equation} \label{ODE}
\dot x = X _0(x)
\end{equation}
with a (stable) limit cycle, there is an analytic local diffeomorphism
$K$, in particular a local change of variables, defined from
$\mathbb{T} \times [-1,1]$ to $\mathbb{R}^2$, a frequency $\omega _0 >
0$ and a rate $\lambda _0 < 0$ such that
\begin{equation}
\label{eq.unperturbed-inv}
 X _0 \circ K(\theta, s) = (\omega _0 \partial _\theta + \lambda _0s
 \partial _s)K(\theta, s) = DK(\theta, s)
\begin{pmatrix}
\omega _0 \\ \lambda _0 s
\end{pmatrix}.
\end{equation}

Hence, if $\theta$ and $s$ satisfy the very simple ODE
\begin{equation}\label{simple} 
\begin{split} 
&\dot \theta(t) = \omega _0,   \\
& \dot s(t) = \lambda _0 s(t), 
\end{split}
\end{equation}
then 
\begin{equation*} 
 x(t)=K(\theta(t),s(t))
\end{equation*} 
is a solution of \eqref{ODE} in a neighborhood of the limit cycle.

Therefore, the paper \cite{Huguet2013} trades finding all the
solutions near the limit cycle of \eqref{ODE} for finding, $K$, $
\omega_0$ and $\lambda_0$ solving \eqref{eq.unperturbed-inv}.  The
paper \cite{Huguet2013} also develops efficient algorithms for the
study of \eqref{eq.unperturbed-inv}, the so-called invariance
equation.

The key idea of the formalism in \cite{Jiaqi2020} consists in
accommodating the delay by just changing the equation
\eqref{eq.unperturbed-inv}. We will obtain a modified functional
equation, which involves non-local terms that reflect the delay in
time.  This equation was treated in \cite{Jiaqi2020}. Hence, we will
produce a two dimensional family of solutions of the delay problem
which resemble the solutions of the unperturbed problem \eqref{ODE}.

The solutions we construct are analogues for the SDDE of the limit
cycle as well as the solutions that converge to the limit cycle
exponentially fast (notice that for the simple ODE, these are all the
solutions with initial data in a neighborhood of the limit cycle).

The set $I_{\theta_0} = \{ K(\theta_0, s_0) \colon s_0 \in [-1,1]\}$
is called in the biology literature the \emph{``isochron''} of
$\theta_0$ because the orbit of a point in $I_{\theta_0}$ converges to
the limit cycle with a phase $\theta_0$. See \cite{Winfree}.

\begin{remark}
The theory of normally hyperbolic manifolds shows that the isochrons
are the same as the stable manifolds of points \cite{Guckenheimer}
(see also \cite{ChiconeL04} for generalizations beyond normal
hyperbolicity).

Therefore, in the ODE case, isochrons and stable manifolds can be used
interchangeably. In the SDDE case, however, the stable manifolds are
infinite dimensional objects.  The solutions we construct are finite
dimensional families.  To avoid confusion with the stable manifolds in
\cite{HaleVL}, we prefer to maintain the name isochrons to refer to
the solutions we construct.  Thus, the isochrons we constuct are
subsets of the (infinite dimennsional) manifolds constructed in
\cite{HaleVL}.  As a matter of fact, they are slow manifolds, they
correspond to the least stable eigenvalues.  One expects that, in
applications, the isochrons will be the most observable solutions
since they correspond to the modes that decrease the slowest so that
any solution will converge to the isochron much faster than the
isochron converges to the limit cycle (an analogue with what happens
in ODE in a stable node).
\end{remark}

\subsection{The perturbed problem} 
\label{sec.overview-perturbed}
We consider now a perturbation of \eqref{ODE} of the form
\begin{equation}\label{eq.perturbed-SDDE}
  \begin{split}
    \dot x(t) =& X\bigl(x(t), \varepsilon x(t - r(x))\bigr)\\
     \coloneq &
    X(x(t),0) + \varepsilon P(x(t), x(t-r(x)),\varepsilon)
    \end{split}
\end{equation}
where $0 < \varepsilon \ll 1$, $X(x(t),0) = X _0(x(t))$ and
$\varepsilon P(x(t), x(t-r(x)),\varepsilon) \coloneq X\bigl(x(t),
\varepsilon x(t - r(x))\bigr)- X(x(t),0)$ and the function $r$ is
positive and as smooth as we need, hence bounded in compact sets.

The equation \eqref{eq.perturbed-SDDE} is a state-dependent delay
differential equation (SDDE) for $\varepsilon \ne 0$. For
typographical reasons we will denote $\widetilde x(t) \coloneq x(t -
r(x))$.

\subsection{The invariance equation in the perturbed problem}
\label{sec.overview-inv}

Let $\tilde{\mathbb{T}}$ be the universal cover of the 1-dimensional
torus $\mathbb{T}$ and let us consider $K\colon
\tilde{\mathbb{T}}\times [-1,1]\rightarrow \mathbb{R}^2$, the
frequency as $\omega _0$ and the rate as $\lambda _0$ which solve
\eqref{eq.unperturbed-inv}. They correspond to the case
$\varepsilon=0$ in \eqref{eq.perturbed-SDDE}.

In analogy with the ODE case, we want to find a $W(\theta, s)$ with
periodicity in the first variable and numbers $\omega $ and $ \lambda$
such that for all $\theta $ and $ s$,
\begin{equation}\label{parameterization}
 x(t) = K\circ W(\theta+\omega t, se^{\lambda t})
\end{equation}  
is a solution of~\eqref{eq.perturbed-SDDE}.

The mapping $W$ gives us a parameterization of the limit cycle with
its isochrons via $K\circ W(\theta,s)$.  That is, the limit cycle will
be represented by the set $\{K \circ W (\theta, 0)\colon \theta \in
\mathbb{T} \}$ and the isochron associated to the angle $\theta $ in
$\mathbb{T}$ will be $\{K\circ W (\theta, s)\colon s \in [-s _0, s
  _0]\}$ where $s _0$ denotes a region of validity in $s$ in the
solution \eqref{parameterization}.

Note that, heuristically (and it is also shown in \cite{Jiaqi2020})
$W$ is close to the identity map and $\omega $ and $ \lambda$ are
close to the values in the unperturbed case.  Hence, we will produce a
two-dimensional family of solutions of the delayed equation
\eqref{eq.perturbed-SDDE} which resemble the solutions of the ODE.

\begin{remark}
Since the phase space of the delay equation is infinite dimensional,
we expect that there are many more solutions of
\eqref{eq.perturbed-SDDE}.

Based on the theory of \cite[Chapter 10]{HaleVL}, we know that the
periodic orbit in the phase space of the SDDE is locally unique, hence
the solution we produce has to agree with the periodic solution
produced in \cite[Theorem 4.1]{HaleVL}.

The theory of delay equations \cite[Chapter 10]{HaleVL} (specially
Theorem 3.2) produces stable (or strong stable) manifolds in the
(infinite dimensional) phase space.  The stable manifolds produced in
\cite{HaleVL} are infinite dimensional. Note that, since the evolution
operators are compact, most of the eigenvalues of the evolution are
very small, in particular, smaller than the $\lambda$ we will select
later, so that the solutions in the stable manifold converge to the
space of solutions produced here in a very fast way.

\end{remark}

Imposing that the tuple $(W, \omega, \lambda)$ is such that
\eqref{parameterization} is a solution of \eqref{eq.perturbed-SDDE}
and knowing that the tuple $(K, \omega _0, \lambda _0)$ is also a
solution of \eqref{eq.unperturbed-inv} but for $\varepsilon = 0$, then
\begin{equation}
\label{eq.inv-deduction}
 DK \circ W D W  = 
 DK \circ W 
 \begin{pmatrix}
  \omega _0 \\ \lambda _0 W _2
 \end{pmatrix} + 
 \varepsilon P (K \circ W, K \circ \widetilde W, \varepsilon),
\end{equation}
where $W _2$ refers to the second component of $W$.

Now, since $K$ is a local diffeomorphism, it also acts as a change of
variable. In particular, we can premultiply \eqref{eq.inv-deduction}
by $(DK\circ W)^{-1}$ to get the functional equation, whose unknowns
are $W\equiv (W _1,W _2)$, $\omega$ and $\lambda$. That functional
equation will be called invariance equation,
\begin{equation}\label{eq.inv}
(\omega \partial _\theta  + \lambda s \partial _s)W(\theta, s) = 
\begin{pmatrix}
\omega _0 \\ \lambda _0 W _2(\theta, s)
\end{pmatrix} + \varepsilon Y\bigl(W(\theta, s), \widetilde W(\theta, s), 
\varepsilon\bigr),
\end{equation}
where we use the shorthand
\begin{equation*} 
\label{abreviations} 
\begin{split} 
 \widetilde W(\theta,s) &\coloneq W\bigl(\theta - \omega r(K\circ W),
 s e^{-\lambda r(K \circ W)}\bigr), \\ Y\bigl(W(\theta, s), \widetilde
 W(\theta, s), \varepsilon\bigr) &\coloneq (DK\circ W(\theta, s))^{-1}
 P\bigl(K\circ W(\theta, s), K\circ \widetilde W(\theta, s),
 \varepsilon\bigr).
\end{split} 
\end{equation*} 

The equation \eqref{eq.inv} will be the center of our attention. Let
us start by making some preliminary remarks on it.

We have ignored the precise definition of the domain of the function
$W$.  We need the range of $W$ to be contained in the domain of
$K$. Note also that it is not clear that the domain of the RHS can
match the domain of the LHS of \eqref{eq.inv}. As it turns out, this
will not matter much for our treatment providing a small enough
$\varepsilon$ (see \cite{Jiaqi2020} for a detailed discussion).

The equation \eqref{eq.inv} is underdetermined. That means, if $W,
\omega$ and $\lambda$ solve equation \eqref{eq.inv}, then $W_{\sigma,
  \eta}, \omega$ and $\lambda$ also solve the same equation with
\begin{equation} \label{eq.shift} 
W_{\sigma, \eta}(\theta, s) = W(\theta + \sigma, \eta s).
\end{equation} 
The parameter $\sigma$ and $\eta$ correspond respectively to choosing
a different origin in the angle coordinate $\theta$ and a different
scale of the parameter $s$.

Even if all these solutions in \eqref{eq.shift} are mathematically
equivalent, we anticipate that choosing a different $\eta$ can change
the reliability of the numerical algorithms.

In \cite{Jiaqi2020}, it is shown that the solutions in the family
\eqref{eq.shift} are locally unique.  That is, all the solutions of
the invariance equation \eqref{eq.inv} are included in
\eqref{eq.shift}.  In equations \eqref{eq.nc1} and \eqref{eq.nc2} we
introduce normalization conditions that specify the parameters in
\eqref{eq.shift}. This is useful for numerics since it allows to
compare easily solutions obtained in different runs with different
discretization sizes.

\begin{remark}
From the point of view of analysis, one of the main difficulties of
the equation \eqref{eq.inv} is that it involves a function composed
with itself (hence the operator is not really differentiable).  Also
the term $\widetilde W$ does not have the same domain as $W$.  We
refer to \cite{Jiaqi2020} for a deeper discussion in the composition
domain.

Similar problems appear in the treatment of center manifolds
\cite{Lanford73, Carr} and indeed, in \cite{Jiaqi2020} there are only
results for finite differentiable solutions and the solutions obtained
may depend on cut-offs and extensions taken to solve the
problem.\footnote{On the other hand, the coefficients of the expansion
  in powers of $s$ are unique and do not depend on cut-offs and
  extensions.}

Based on the experience with center manifolds, we believe that indeed,
the solutions could only be finitely differentiable and that there are
different solutions of the invariance equation (depending on the
extensions considered).
\end{remark}

\begin{remark} 
In the language of ergodic theory, for those familiar with it, the
results of \cite{Jiaqi2020} can be described as saying that there is a
\emph{factor} in the (infinite dimensional) phase space of the SDDE
which is a two dimensional flow with dynamics close to the dynamics of
the ODE.

In this paper, we will compute numerical approximations of the map
giving the semiconjugacy as well as the new dynamics of such a
\emph{factor}.
\end{remark}

\subsection{Format of solution for the invariance equation \eqref{eq.inv}} 
It is shown in \cite{Jiaqi2020} that, for small $\varepsilon$, one can
construct smooth solutions of \eqref{eq.inv} of the form
\begin{equation}
\label{eq.W-power-s}
 W(\theta, s) = W^0(\theta) + W^1(\theta) s + \sum _{j = 2} ^{n}
 W^j(\theta) s^j + W^>(\theta, s)
\end{equation}
where $W ^j \colon \mathbb{T} \rightarrow \mathbb{T}\times \mathbb{R}$
and $W^> \colon \mathbb{T} \times [-s _0, s _0] \rightarrow
\mathbb{T}\times \mathbb{R}$ with $W^>(\theta, s) = O(s^{n+1})$ and
for some $s _0 > 0$.

As we will see in more detail later on, if one substitutes
\eqref{eq.W-power-s} into \eqref{eq.inv} and matches powers in $s$,
one gets a finite set of recursive equations for the coefficients
$W^j$ of the expansion (and for $\omega$, $\lambda$).  We will deal
with these equations in detail later.  Note that this will require a
discretization of $W^j$, which are only functions of the angle
$\theta$.

\subsection{The equations for terms of the expansion of \texorpdfstring{$W$}{W}}
\label{sec.expansion}
Assume for the moment that $W^0$ and $ W^1$ have already been
computed. Then we can substitute the expansion \eqref{eq.W-power-s} of
$W$ in powers of $s$ into the invariance equation
\eqref{eq.inv}. Matching the coefficients of the powers of $s$ on both
sides we obtain a hierarchy of equations for $W^j$, $j = 2,3,\ldots $.

The equations for $W^j$ involve just $W^0, \dotsc, W^{j-1}$. Hence,
they can be studied recursively.  In \cite{Jiaqi2020} it is shown
that, if we know $W^0,\dotsc, W^{j -1}$, it is possible to find $W^j$
in a unique way and, hence, we can proceed to solve the equations
recursively.  In this paper, we show that there are precise algorithms
to compute these recursions.  We also report results of implementation
in some cases.

Note that $W^j$, $j = 0,\ldots, n$ in \eqref{eq.W-power-s} are
functions only of $\theta$. The function $W^>$ depends both on
$\theta$ and $s$ but vanishes at high order in $s$ and does not enter
in the equations for $j = 0, \dotsc, n$.

As it frequently happens in perturbative expansions, the low order
equations are special.  The equation for $W^0$ -- which gives the
periodic solution that continues the limit cycle -- also determines
$\omega$.  The equation for $W^1$ also determines $\lambda$. The
equations for $W^j$, $j = 2,\dotsc, n$ are all similar and involve
solving the same operator (with different terms).

In this paper, we will only consider the computation of the $W^j$, $j
=0, 1, \ldots $.  The term $W^{>}$, is estimated in \cite{Jiaqi2020}
and it is not only high order in $s$ but also actually small in rather
large balls.

Note that even if the $W^j$ are unique (up to the parameters in
\eqref{eq.shift}), the $W^{>}$ depends on properties of the extension
considered.  This is, of course very reminiscent of what happens in
the theory of center manifolds \cite{Sijbrand85} .  For numerical
studies of expansions of center manifolds we refer to \cite{BeynK98,
  Jorba99,PotzscheR06} and detailed estimates of the truncation in
\cite{CapinskiR12}.  The numerical considerations about the effect of
the truncation apply with minor changes to our case.

\subsection{Periodicity conditions} \label{sec.overview-pc}

From the point of view of implementation in computers, it is
convenient to think of the functions $K$ and $W$ in \eqref{eq.inv}
which involve angle variables (and which range on angles), as real
functions with boundary conditions (in mathematical language this is
described as taking lifts).  Hence, we take
\begin{equation}\label{eq.pc} 
\begin{split}
 K(\theta + 1, s) &= K(\theta, s), \\
 W(\theta + 1, s) &= W(\theta, s) +
 \begin{pmatrix}1\\ 0 \end{pmatrix}.  
\end{split} 
\end{equation}

Notice that we are normalizing the angles to run between $0$ and
$1$. Very often, the angles are taken to run in $[0, 2 \pi)$.

The periodicity conditions in \eqref{eq.pc} indicate the second
component of $W$ is periodic (it describes a radial coordinate) in
$\theta$ while the first component increases by $1$ when $\theta$
increases by $1$ (it describes an angle). So that the circle described
by increasing $\theta$ makes the angle in the coordinate go around, so
that it is a non-contractible circle in the angle.

For the expansion of $W$ in powers of $s$ as in \eqref{eq.W-power-s},
the periodicity conditions amount to:

\begin{equation}\label{eq.pc-W}
W^0(\theta+1) = W^0(\theta) + 
\begin{pmatrix}
1 \\ 0
\end{pmatrix} \quad \text{and} \quad 
W^j(\theta+1) = W^j(\theta)  \quad \text{for }j\geq 1.
\end{equation}

In the numerical analysis, there are many well-known ways to
discretize periodic functions.  We will use Fourier series, but there
are also other alternatives such as periodic splines.

In general, for functions $\Psi$ with $\Psi(\theta +1) =
\Psi(\theta)+1$, we define ${\widetilde \Psi}(\theta) = \Psi(\theta) -
\theta$ which is a periodic function, i.e. for all $\theta$,
${\widetilde \Psi}(\theta+1) = {\widetilde \Psi}(\theta)$. Then we
will discretize $\widetilde \Psi$ and rewrite the functional equations
so that this is the only unknown.

\subsection{Normalization of the solutions} \label{sec.overview-nc}
As indicated in the discussion, the invariance equation has two
obvious sources of indeterminacy: One is the choice of the origin of
the variable $\theta$ (the $\sigma$ in \eqref{eq.shift}) and the other
is the choice of the scale of the variable $s$ (the $\eta$ in
\eqref{eq.shift}) .  In \cite{Jiaqi2020} it is shown that these are
the only indeterminacies and that once we fix them, we can get any
other solution by applying \eqref{eq.shift}.

A convenient way to fix the origin of $\theta$ is to require
\begin{equation} \label{eq.nc1}
\int_0^1 \bigl[ \partial_\theta W^0_1 (\theta,0) \, W_1(\theta, 0)
  \bigr] \, d \theta = a
\end{equation}
where $W^0$ is an initial approximation and $a$ is a real number,
typically it is closed to $1$.  This normalization is easy to compute
and is rather sensitive since, when we move in the family
\eqref{eq.shift}, the derivative with respect to the shift is a
positive number.

The normalization of the origin of coordinates, has no numerical
consequences except for the possibility of comparing the solutions in
different runs. The solutions corresponding to different
normalizations have very similar properties. The numerical algorithm
\ref{alg.s0} in its step \ref{alg.s0-pc} leads to a small drift in the
normalization in each iteration, but it is guaranteed to converge to
one of the solutions in \eqref{eq.shift}.

The second normalization is just a choice of the eigenvector of an
operator. We have found it convenient to take
\begin{equation}\label{eq.nc2}
\int_{0}^{1} \partial _s W _2(\theta, 0)\, d\theta = \rho
\end{equation}
with a real $\rho \ne 0$.

We anticipate that changing the value of $\rho$ is equivalent to
changing $s$ into $b s$ where $b$ is commonly named \emph{scaling
  factor}.

All the choices of $\rho$ are mathematically equivalent -- they amount
to setting the scale of the parameter $s$ --.  The choice of this
normalization, however, affects the numerical accuracy
dramatically. Notice that if we change $s$ into $b s$, the
coefficients $W ^j(\theta)$ in \eqref{eq.W-power-s} change into $W
^j(\theta) b ^j$. So, different choices of $b $ may lead the Taylor
coefficients to be very large or very small, which makes the
computations with them very susceptible to round off error. It is
numerically advantageous to choose the scale in such a way that the
Taylor coefficients have a comparable size.

In practice, we run the calculations twice. A preliminary one whose
only purpose is to compute an approximation of the scale that makes
the coefficients to remain more or less of the same size. Then, a more
definitive calculation can be run. That last running is more
numerically reliable.

\begin{remark}\label{rem.newton-subtle}
In standard implementation of the Newton Method for fixed points of a
functional, say $\Psi$, the fact that the space of solutions is two
dimensional leads $D\Psi - \text{Id}$ to have a two dimensional
kernel, and not being invertible.

In our case, we will develop a very explicit and fast algorithm that
produces an approximate linear right inverse. This linear right
inverse leads to convergence to an element of the family
\eqref{eq.shift}.
\end{remark}

\section{Computation of \texorpdfstring{$(K,\omega _0, \lambda _0)$ --}{(K,omg0, lam0) -} unperturbed case}
\label{sec.unperturbed-case}

For completeness, we quote the Algorithm~4.4 in~\cite{Huguet2013}
adding some practical comments. That algorithm allows us to
numerically compute $\omega _0$, $\lambda _0$ and $K\colon
\tilde{\mathbb{T}}\times [-1,1] \rightarrow \mathbb{R}^2$ in
\eqref{eq.inv}. We note that the algorithm has quadratic convergence
as it was proved in \cite{Huguet2013}.

\begin{alg}\label{alg.huguet} Quasi-Newton method \
\begin{enumerate}
\setlength{\itemsep}{.9em}
\renewcommand*{\theenumi}{\emph{\arabic{enumi}}}
\renewcommand*{\labelenumi}{\theenumi.}
\item [$\star$] \texttt{Input:} $\dot x=X(x)$ in $\mathbb{R}^2$,
  $K(\theta,s) = K ^0(\theta) + K ^1(\theta)b _0s$, $\omega _0> 0$,
  $\lambda _0 \in \mathbb{R}$ and scaling factor $b _0>0$.
\item [$\star$] \texttt{Output:} $K(\theta,s)=\sum_{j=0}^{m-1} K
  ^j(\theta) (b _0s)^j$, $\omega _0$ and $\lambda _0$ such that
  $\lVert E \rVert \ll 1$.
\item\label{alg.huguet-loop} $E \gets X\circ K - (\omega _0 \partial
  _\theta + \lambda _0 s \partial _s)K$.
\item\label{alg.huguet-linear-system} Solve $DK \tilde E = E$ and
  denote $\tilde E \equiv (\tilde E _1, \tilde E _2)$.
\item $\sigma \gets \int _0^1 \tilde E _1(\theta, 0)\, d\theta$ and
  $\eta \gets \int _0^1 \partial _s \tilde E _2(\theta, 0)\, d\theta$.
\item $E _1 \gets \tilde E _1 - \sigma $ and $E _2 \gets \tilde E _2 -
  \eta s$.
\item\label{alg.huguet-cohom1} Solve $(\omega _0 \partial
  _\theta + \lambda _0 s \partial _s)S _1=E _1$ imposing
  \begin{equation}\label{eq.huguet-cohom1-cond}
    \int _0^1 S
    _1(\theta, 0)\, d\theta = 0.
    \end{equation}
\item\label{alg.huguet-cohom2} Solve $(\omega _0 \partial
  _\theta + \lambda _0 s \partial _s)S _2 - \lambda _0 S _2=E _2$
  imposing
\begin{equation}  \label{eq.huguet-cohom2-cond}
\int _0^1 \partial _s S _2(\theta, 0)\, d\theta = 0.
\end{equation}
\item $S \equiv (S _1, S _2)$.
\item Update: $K\gets K+ DK S$, $\omega _0 \gets \omega _0 + \sigma $
  and $\lambda _0 \gets \lambda _0 + \eta$.
\item Iterate~\eqref{alg.huguet-loop} until convergence in $K$,
  $\omega$ and $\lambda$. Then undo the scaling $b _0$.
\end{enumerate}
\end{alg}
Algorithm~\ref{alg.huguet} requires some practical considerations:
\begin{enumerate}
\renewcommand*{\theenumi}{\emph{\roman{enumi}}}
\renewcommand*{\labelenumi}{\theenumi.}
\item It is clear that the most delicate steps of above algorithm are
  \ref{alg.huguet-cohom1} and \ref{alg.huguet-cohom2}, which are often
  called \emph{cohomology equations}. These steps involve solving
  PDE's whereas the others are much simpler.  Indeed, the
  discretizations used are dictated by the desire of solving these
  equations in an efficient way.

\item \emph{Initial guess.} $K ^0\colon \tilde{\mathbb{T}}\rightarrow
  \mathbb{R}^2$ will be a parameterization of the periodic orbit of
  the ODE with frequency $\omega _0$. It can be obtained, for
  instance, by a Poincar\'e section method, continuation of integrable
  systems or Lindstedt series.  An approximation for $K ^1\colon
  \tilde{\mathbb{T}}\rightarrow \mathbb{R}^2$ and $\lambda$ can be
  obtained by solving the variational equation
  \begin{align*}
    DX\circ K ^0(\theta) U(\theta) &= \omega _0 \frac{d}{d\theta}
    U(\theta), \\ U(0) &= Id _2.
  \end{align*}
  Hence if $(e^{\lambda _0 / \omega _0}, K ^1(0))$ is the eigenpair of
  $U(1)$ such that $\lambda _0 < 0$, then $K ^1(\theta) = U(\theta) K
  ^1(0) e ^{-\lambda _0 \theta / \omega _0}$.
\item \emph{Stopping criteria.} As any Newton method, a possible
  condition to stop the iteration can be when either $\lVert E \rVert$
  or $\max\{\lVert DK S \rVert, \lvert \sigma \rvert, \lvert
  \eta\rvert \}$ is smaller than a given tolerance.

  Note that the \emph{a posteriori} theorems in \cite{Huguet2013} give
  a criterion of smallness on the error depending on properties of the
  function $K$. If these criteria are satisfied, one can ensure that
  there is a true solution close to the numerical one.

\item \emph{Uniqueness.}  Note that in the steps
  \ref{alg.huguet-cohom1} and \ref{alg.huguet-cohom2}, which involve
  solving the cohomology equations, the solutions are determined only
  up to adding constants in the zero or first order terms. We have
  adopted the conventions \eqref{eq.huguet-cohom1-cond},
  \eqref{eq.huguet-cohom2-cond}. These conventions make the solution
  operator linear (which matches well the standard theory of
  Nash-Moser methods since it is easy to estimate the norm of the
  solutions).

As it is shown in \cite{Huguet2013}, the algorithm converges
quadratically fast to a solution, but since the problem is
underdetermined, we have to be careful when comparing solutions of
different discretization. In \cite{Huguet2013} there is discussion of
the uniqueness, but for our purposes in this paper, any of the
solutions will work.  The uniqueness of the solutions considered in
this paper is discussed in section~\S\ref{sec.overview-nc}.

\item \emph{Convergence.} It has been proved in~\cite{Huguet2013} that
  even of the quasi-Newton method, it still has quadratic convergence.

Note that it is remarkable that we can implement a Newton like method
without having to store -- much less invert -- any large matrix. Note
also that we can get a Newton method even if the derivative of the
operator in the fixed point equation has eigenvalues $1$. See
remark~\ref{rem.newton-subtle}.
\end{enumerate}

\newcommand{\pfun}{S}
\subsection{Fourier discretization of periodic functions} 
As it was mentioned before, the key step of Algorithm~\ref{alg.huguet}
is to solve the equations in steps \ref{alg.huguet-cohom1} and
\ref{alg.huguet-cohom2}. Their numerical resolution will be
particularly efficient when the functions are discretized in
Fourier-Taylor series. This will be the only discretization we will
consider in this paper providing a deep discussion.

\begin{remark} 
Even if we will not use it in this paper, we remark that \cite[\S
  4.3.1]{Huguet2013} there are two methods to solve them. One assumes
a Fourier representation in terms of the angle $\theta$ and the other
uses integral expressions which can be evaluated efficiently in
several discretizations of the functions (e.g. splines). The spline
representation could be preferable to the Fourier-Taylor in some
regimes where the limit cycles are bursting.
\end{remark} 

Recall that a function $\pfun \colon \mathbb{R} \rightarrow
\mathbb{R}$ is called periodic when $\pfun (\theta+1)=\pfun (\theta)$
for all $\theta$.

To get a computer representation of a periodic function, we can either
take a mesh in $\theta$, i.e. $(\theta _k) _{k=0}^{n_\theta - 1}$ and
store the values of $\pfun $ at these points: $\widecheck \pfun =
(\widecheck \pfun _k) _{k=0}^{n_\theta - 1} \in \mathbb{R}^{n
  _\theta}$ with $\widecheck \pfun _k = \pfun(\theta _k)$ or we can
take advantage of the periodicity and represent it in a trigonometric
basis.

The Discrete Fourier Transform (DFT), and also its inverse, allows to
switch between the two representations above.  If we fix a mesh of
points of size $n _\theta$ uniformly distributed in $[0,1)$,
  i.e. $\theta _k = k / n_\theta$, the DFT is:
\[
 \widehat \pfun = (\widehat \pfun _k)_{k=0}^{n _\theta - 1} \in
 \mathbb{C}^{n _\theta}
\]
so that 
\begin{equation}\label{eq.dft-top}
 \widecheck \pfun _k = \sum _{j = 0}^{n _\theta - 1} \widehat \pfun _j
 e^{2\pi \I j k / n _\theta}
\end{equation}
or equivalently
\begin{equation}\label{eq.dft-toc}
 \widehat \pfun _k = \frac{1}{n _\theta} \sum _{j = 0}^{n _\theta - 1}
 \widecheck \pfun _j e^{-2\pi \I j k / n _\theta}.
\end{equation}

In the case of a real valued function, $\widehat \pfun _0$ is real and
the complex numbers $\widehat \pfun $ satisfy Hermitian symmetry,
i.e. $\widehat \pfun _k = \widehat \pfun _{n_\theta -k}^\ast$
(denoting by ${}^\ast$ the complex conjugate), which implies $\widehat
\pfun _{n _\theta / 2}$ real when $n _\theta$ is even.  Then, we
define real numbers $(a_0; a _k, b _k)_{k=1}^{\lceil n _\theta/2\rceil
  -1}$ if $n _\theta$ is odd, here $\lceil \cdot \rceil$ denotes the
ceil function, otherwise $(a_0, a_{n _\theta/2}; a _k, b _k)_{k=1}^{n
  _\theta/2-1}$ defined by
\[
 a _0 = 2 \widehat \pfun _0\text{, }a _{n _\theta/2} = 2 \widehat
 \pfun _{n _\theta/2}\text{, }a _k = 2 \Re \widehat \pfun _k \text{
   and } b _k = -2\Im \widehat \pfun _{k}
\]
with $1 \leq k < \lceil n _\theta/2 \rceil$. 

Thus, $\pfun $ can be approximated by
\begin{equation}\label{eq.dft-top-real}
 \pfun (\theta) = \frac{a _0}{2} + \frac{a _{n _\theta /2}}{2} \cos
 (\pi n _\theta \theta) + \sum _{k = 1}^{\lceil n _\theta /2 \rceil
   -1} a _k \cos (2 \pi k \theta) + b _k \sin(2\pi k \theta)
\end{equation}
where the coefficient $a _{n _\theta/2}$ only appears when $n _\theta$
is even and it refers to the aliasing notion in signal theory.

Therefore \eqref{eq.dft-top-real} is equivalent to \eqref{eq.dft-top}
but rather than $2n _\theta$ real numbers, only half of them are
needed.

Henceforth, all real periodic functions $\pfun$ can be represented in
a computer by an array of length $n _\theta$ whose values are either
the values of $\pfun$ on a grid or the Fourier coefficients.  These
two representations are, for all practical purposes equivalent since
there is a well known algorithm, Fast Fourier Transform (FFT), which
allows to go from one to the other in $\Theta(n _\theta \log n
_\theta)$ operations.  The FFT has very efficient implementations so
that the theoretical estimates on time are realistic (we can use
\textsc{fftw3} \cite{FFTW2005}, which optimizes the use of the
hardware).

We can also think of functions of two variables $W(\theta,s)$ where
one variable $\theta$ is periodic and the other variable $s$ is a real
variable.  In the numerical implementations, the variable $s$ will be
discretized as a polynomial. Thus $W(\theta, s)$ can be thought as a
function of $\theta$ taking values in polynomials of length
$n_s$. Hence, a function of two variables with periodicity as above
will be discretized by an array $n_\theta \times n_s$.  The meaning
could be that it is a polynomial for each value of $\theta$ in a mesh
or that it a polynomial of whose coefficients are Fourier
coefficients.  Alternatively, we could think of $W(\theta, s)$ as a
polynomial in $s$ taking values in a space of periodic functions.

This mixed representation of Fourier series in one variable and power
series in another variable, is often called Fourier-Taylor series and
has been used in celestial mechanics for a long time, dates back to
\cite{Broucke} or earlier.  We note that, modern computer languages
allow to overload the arithmetic operations among different types in a
simple way.

It is important to note that all the operations in
Algorithm~\ref{alg.huguet} are fast either on the Fourier
representation or in the values of a mesh representation.  For
example, the product of two functions or the composition on the left
with a known function are fast in the representation by values in a
mesh. More importantly for us, as we will see, the solution of
cohomology equations is fast in the Fourier representation.  On the
other hand, there are other steps of Algorithm~\ref{alg.huguet}, such
as adding, are fast in both representations.

Similar consideration of the efficiency of the steps will apply to the
algorithms needed to solve our problem. The main novelty of the
algorithms in this paper compared with those of \cite{Huguet2013} is
that we will need to compose some of the unknown functions (in
\cite{Huguet2013} the unknowns are only composed on the left with a
known function). The algorithms we use to deal with composition will
be presented in section~\S\ref{sec.composing-periodic-maps}.  The
composition operator will be the most delicate numerical aspect, which
was to be expected, since it was also the most delicate step in the
analysis in \cite{Jiaqi2020}. The composition operator is analytically
subtle. A study which gives examples that results are sharp is in
\cite{Obaya1999}. See also \cite{AppellZ90}.

\begin{remark} 
Fourier series are extremely efficient for smooth functions which do
not have very pronounced spikes. For rather smooth functions -- a
situation that appears often in practice -- it seems that Fourier
Taylor series is better than other methods.

It should be noted, however that in several models of interest in
electronics and neuroscience, the solutions move slowly during a large
fraction of the period, but there is a fast movement for a short time
(bursting).  In these situations, the Fourier scheme has the
disadvantage that the coefficients decrease slowly and that the
discretization method does not allow to put more effort in describing
the solutions during the times that they are indeed changing
fast. Hence, the Fourier methods become unpractical when the limit
cycles are bursting.  In such cases, one can use other methods of
discretization. In this paper, we will not discuss alternative
numerical methods, but note that the theoretical estimates of
\cite{Jiaqi2020} remain valid independent of the method of
discretization. We hope to come back to implementing the toolkit of
operations of this paper in other discretizations.
\end{remark}

\begin{remark} 
One of the confusing practical aspects of the actual implementation is
that the coefficients of the Fourier arrays are often stored in a
complicated order to optimize the operations and the access during the
FFT.

For example concerning the coefficients $a _k$'s and $b _k$'s in
\eqref{eq.dft-top-real}, in \textsc{fftw3}, the
\texttt{fftw\_plan\_r2r\_1d} uses the following order of the Fourier
coefficients in a real array $(v _0, \dotsc, v _{n _\theta-1})$.
\begin{align*}
 v _0 &= a _0, \\ v _k &= 2a _k \text{ and } v _{n _\theta - k} =-2b
 _k \quad \text{ for } 1 \leq k < \lceil n _\theta/2 \rceil, \\ v
 _{n _\theta/2} &= a _{n _\theta/2}
\end{align*}
where the index $n _\theta/2$ is taken into consideration if and only
if $n _\theta$ is even. Another standard order in other packages is
just $(a _0, a _{n _\theta/2}; a _k, b _k)$ in sequential order or $(a
_0; a _k, b _k)$ if $n _\theta$ is odd.
\end{remark}

To measure errors and size of functions represented by Fourier series,
we have found useful to deal with weighted norms involving the Fourier
coefficients.
\begin{align*}
\lVert \pfun \rVert _{w\ell^1, n} &= 2(n _\theta/2)^n | \widehat S _{n
  _\theta /2} | + \sum _{k=1}^{\lceil n _\theta/2 \rceil -1} ((n
_\theta-k)^{n} + k ^{n} ) \lvert \widehat \pfun _k \rvert \\ &= (n
_\theta /2)^{n} \lvert a _{n _\theta /2} \rvert + \frac{1}{2} \sum
_{k=1}^{\lceil n _\theta/2 \rceil -1} ((n _\theta-k)^{n} + k ^{n} ) (a
_k^2 + b _k^2)^{1/2} .
\end{align*}
where, again, the term for $n _\theta /2$ only appears if $n _\theta$
is even.

The smoothness of $\pfun$ can be measured by the speed of decay of the
Fourier coefficients and indeed, the above norms give useful
regularity classes that have been studied by harmonic analysts.

\begin{remark} 
The relation of the above regularity classes with the the most common
$C^m$ is not straightforward, as it is well known by Harmonic
analysts, \cite{Stein70}.

Riemann-Lebesgue's Lemma tells us that if $\pfun$ is continuous and
periodic, $\widehat \pfun _k \to 0$ as $k \to \infty$ and in general
if $\pfun$ is $m$ times differentiable, then $\lvert \widehat \pfun _k
\rvert |k|^m $ tends to zero. In particular, $\lvert \widehat \pfun _k
\rvert \leq C / \lvert k \rvert ^m$ for some constant $C>0$.

In the other direction, from $\lvert \widehat \pfun _k \rvert \leq C /
\lvert k \rvert ^m$ we cannot deduce that $\pfun \in C^m$.

One has to use more complicated methods.  In \cite{Petrov2002} it was
found that one could find a practical method based on Littlewood-Paley
theorem (see \cite{Stein70}) which states that the function $\pfun$ is
in $\alpha$-H\"older space with $\alpha \in \mathbb{R}_+$ if and only
if, for each $\eta \geq 0$ there is constant $C >0$ such that for all
$t >0$.
\[
\biggl\lVert \biggl(\frac{\partial}{\partial t} \biggr)^\eta e ^{-t
  \sqrt{-\Delta}\theta} \biggr\rVert _{L^\infty (\mathbb{T})} \leq C t
^{\alpha - \eta}.
\]
The above formula is easy to implement if one has the Fourier
coefficients, as it is the case in our algorithms.
\end{remark}

\goodbreak

\subsection{Solutions of the cohomology equations in Fourier 
representation} 

Under the Fourier representation we can solve the cohomological
equations in the steps \ref{alg.huguet-cohom1} and 
\ref{alg.huguet-cohom2} of the Algorithm~\ref{alg.huguet}.

\begin{pro}[Fourier version, \cite{Huguet2013}]
  \label{pro.Fourier-solution} 
Let $E(\theta, s)=\sum _{j,k} E _{jk} e^{2\pi \I k \theta} s^j$.
\begin{itemize}
\item If $E_{00}=0$, then $(\omega \partial _\theta + \lambda s
  \partial _s) u(\theta, s) = E(\theta, s) $ has solution $u(\theta,
  s)= \sum _{j,k} u _{jk} e^{2\pi \I k \theta}s^j$ and
  \[
   u _{jk} =
   \begin{cases}
    \frac{E _{jk}}{\lambda j + 2\pi \I \omega k} & \text{if } (j,k)\ne
    (0,0) \\ \alpha & \text{otherwise}.
   \end{cases}
  \]
  for all real $\alpha$. Imposing $\int _0^1 u(\theta,0)\, d\theta =
  0$, then $\alpha = 0$.
\item If $E _{10}=0$, then $(\omega \partial _\theta + \lambda s
  \partial _s - \lambda) u(\theta, s) = E(\theta, s) $ has solution
  $u(\theta, s)= \sum _{j,k} u _{jk} e^{2\pi \I k \theta}s^j$ and
  \[
   u _{jk} = 
   \begin{cases}
    \frac{E _{jk}}{\lambda (j-1) + 2\pi \I \omega k} & \text{if }
    (j,k)\ne (1,0) \\ \alpha & \text{otherwise}.
   \end{cases}
  \]
  for all real $\alpha$. Imposing $\int _0^1 \partial _s u(\theta,0)\,
  d\theta = 0$, then $\alpha = 0$.
\end{itemize}
\end{pro}

The paper \cite{Huguet2013} also presents a solution in terms of
integrals. Those integral formulas for the solution are independent of
the discretization and work for discretizations such as Fourier
series, splines and collocations methods. Indeed, the integral formulas
are very efficient for discretizations in splines or in collocation
methods.  In this paper we will not use them since we will discretize
functions in Fourier series and for this discretization, the methods
described in Proposition~\ref{pro.Fourier-solution} are more efficient.

\subsection{Treatment of the step \ref{alg.huguet-linear-system} in  Algorithm~\ref{alg.huguet}}

To solve the linear system in the step \ref{alg.huguet-linear-system}
of Algorithm~\ref{alg.huguet}, we can use
Lemma~\ref{lem.poly-linear-system}, whose proof is a direct power
matching.

\begin{lem}\label{lem.poly-linear-system}
Consider the equation for $x$ given by Let $A(\theta,s) x(\theta, s) =
b(\theta, s)$ where $A, b$ are given. More explicitly:
\[
\biggl(\sum_{k\geq 0} A_k(\theta) s^k \biggr) \sum_{k\geq 0} \ve
x_k(\theta)s^k = \sum_{k\geq 0} \ve b_k(\theta) s^k.
\]
Then, the coefficients $\ve x_k(\theta)$ are obtained recursively by
solving
\[
A_0(\theta) \ve x_k(\theta) = \ve b_k(\theta) - \sum_{j=1}^{k}
A_j(\theta) \ve x_{k-j}(\theta).
\]
which can be done provided that $A_0(\theta)$ is invertible and that
one knows how to multiply and add periodic functions of $\theta$.
\end{lem}

We also recall that composition of a polynomial in the left with a
exponential, trigonometric functions, powers, logarithms (or any
function that satisfies an easy differential equation) can be done
very efficiently using algorithms that are reviewed in
\cite{Haro2016} which goes back to \cite{Knuth1981}.

We present here the case of the exponential which will be used later
on in Algorithm~\ref{alg.sn}.

If $P$ is a given polynomial -- or a power series -- with coefficients
$P _j$, we see that $E(s) = \exp P(s)$ satisfies
\[
\frac{d}{ds} E(s) = E(s) \frac{d}{ds} P(s)
\]
Equating like powers on both sides, it leads to $E_0 = \exp P(0)$,
and the recursion:
\[
E_{j} = \frac{1}{j} \sum _{k=0}^{j-1} (j-k) P _{j-k} E _{k}
\text{, } \qquad j \geq 1,
\]
Note that this can also be done if the coefficients of $P$ are
periodic functions of $\theta$ (or polynomials in other variables).
In modern languages supporting overloading or arithmetic functions,
all this can be done in an automatic manner.

Note that if the polynomial has degree $n_s$, the computation up to
degree $n_s$ takes $\Theta( n_s^2)$ operations of multiplications of
the coefficients.

\section{Computation of \texorpdfstring{$(W,\omega, \lambda)$ --}{(W,omg,lamb) -} perturbed case}
\label{sec.perturbed-case}
The main result in the paper \cite{Jiaqi2020} 
states that if $\varepsilon$ in \eqref{eq.inv} is small enough, a
periodicity condition like \eqref{eq.nc1} and a normalization like
\eqref{eq.nc2} are considered, then there exists a unique tuple $(W,
\omega, \lambda)$ verifying \eqref{eq.inv}, \eqref{eq.nc1} and
\eqref{eq.nc2}.

The formulation of that result in \cite{Jiaqi2020} is done in \emph{a
  posteriori} format which ensures the existence of a true solution
once an approximate enough solution is provided as initial guess for
the iterative scheme.

Moreover, it also gives the Lipschitz dependence of the solution on
parameters which allows to consider a continuation approach.

We refer to \cite{Jiaqi2020} for a precise formulation of the result
involving choices of norms to measure the error in the approximate
solutions.

\subsection{Fixed point approach}
We compute all the coefficients $W^j(\theta)$ of the truncated
expression $W(\theta,s)$ in \eqref{eq.W-power-s} order by order.  The
zero and first orders require a special attention due to the fact that
the values $\omega$ and $\lambda$ are obtained in the equation
\eqref{eq.inv} matching coefficients of $s^0$ and $s^1$ respectively.
The condition that allows to obtain $\omega$ comes from the
periodicity condition~\eqref{eq.pc-W}.  The mapping $W^0$ is not a
periodic function. But we can use it to get a periodic one defined by
$\hat W^0(\theta) \coloneq W^0(\theta) - \left(
\begin{smallmatrix}
\theta \\ 0
\end{smallmatrix} \right)$.
The condition for $\lambda$ is given by the normalization
condition~\eqref{eq.nc2}. As in the unperturbed case, we are allowed
to use a scaling factor. The use of such a scaling factor allows to
set the value of $\rho$ in \eqref{eq.nc2} equal to $1$.

Algorithm~\ref{alg.s0} sketches the fixed-point procedure to get
$\omega$ and $W^0$ whose periodicity condition is ensured in step
\eqref{alg.s0-pc}. In this case the initial condition will be $\omega
_0$ (the value for $\varepsilon = 0$) for $\omega$ and
$\left(\begin{smallmatrix} \theta \\ 0\end{smallmatrix} \right)$ for
  $W^0(\theta)$ since $W(\theta,s)$ is close to the identity.

\begin{alg}[$s^0$ case] \label{alg.s0} \ \\
Let $\widetilde{W^0}(\theta)\coloneq W^0\bigl(\theta-\omega r\circ
K(W^0(\theta))\bigr)$.
\begin{enumerate}
\setlength{\itemsep}{.9em}
\renewcommand*{\theenumi}{\emph{\arabic{enumi}}}
\renewcommand*{\labelenumi}{\theenumi.}
\item [$\star$] \texttt{Input:} $\dot x = X(x) + \varepsilon
  P(x,\tilde x, \varepsilon)$, $0<\varepsilon \ll 1$, $K(\theta,s) =
  \sum_{j=0}^{m-1} K ^j(\theta) (b _0s)^j$, $b _0 > 0$, $\omega _0>0$
  and $\lambda _0 < 0$.
\item [$\star$] \texttt{Output:} $\hat W^0\colon \mathbb{T}\rightarrow
  \mathbb{R}^2$ and $\omega > 0$.
\item $\hat W^0(\theta) \gets 0$ and $\omega \gets \omega _0$.
\item \label{alg.s0-loop} $W^0(\theta) \gets
\begin{pmatrix}
\theta \\ 0
\end{pmatrix} + \hat W^0(\theta)$.
\item Solve $DK\circ W^0(\theta) \eta(\theta) = \varepsilon P(K\circ
  W^0(\theta), K\circ \widetilde{W^0}(\theta), \varepsilon)$. Let
  $\eta \equiv (\eta_1, \eta_2)$.
\item $\alpha \gets \int _0^1 \eta_1(\theta)\, d\theta$ and $\omega
  \gets \omega _0 + \alpha$.
\item \label{alg.s0-pc} Solve $\omega \partial _\theta \hat W
  _1^0(\theta) = \eta_1(\theta) - \alpha$ imposing $\int _0^1 \hat W
  _1^0(\theta) \, d\theta=0$.
\item Solve $(\omega \partial _\theta - \lambda _0) \hat W
  _2^0(\theta) = \eta_2(\theta)$.
\item Iterate~\eqref{alg.s0-loop} until convergence in $W^0$ and
  $\omega$.
\end{enumerate}
\end{alg}
Algorithm~\ref{alg.sn} sketches the steps to compute $(W^1,\lambda)$
and $W^n$ for $n \geq 2$. The initial guesses are $\lambda _0$ for
$\lambda$, $ \left(
\begin{smallmatrix}
0 \\ 1
\end{smallmatrix} \right)$ for $W^1$ and $
\left(
\begin{smallmatrix}
0 \\ 0
\end{smallmatrix} \right)$
for $W^n$. In either case, it is required to solve a linear system of
the form of Lemma~\ref{lem.poly-linear-system} as well as
cohomological equation similar to the unperturbed case.
\begin{alg}[$s^1$ case and $s^n$ case with $n\geq 2$] \label{alg.sn} \ \\  
Let $\widetilde W(\theta,s) \coloneq W\bigl(\theta - \omega r \circ
K(W(\theta,s)), s e^{-\lambda r \circ K(W(\theta,s))}\bigr)$.
\begin{enumerate}
\setlength{\itemsep}{.9em}
\renewcommand*{\theenumi}{\emph{\arabic{enumi}}}
\renewcommand*{\labelenumi}{\theenumi.}
\item [$\star$] \texttt{Input:} $\dot x=X(x) + \varepsilon
  P(x,\widetilde x, \varepsilon)$, $0 < \varepsilon \ll 1$, $K(\theta,
  s) = \sum_{j=0}^{m-1} K ^j(\theta)(b _0s)^j$, $b _0>0$, $\omega _0 >
  0$, $\lambda _0 < 0$, $\hat W ^0(\theta)$, $W ^j(\theta)$ for $0 < j
  < n$, $b > 0$ and $\omega > 0$.
\item [$\star$] \texttt{Output:} either $W^1\colon
  \mathbb{T}\rightarrow \mathbb{T} \times \mathbb{R}$ and $\lambda<0$
  or $W^n\colon \mathbb{T}\rightarrow \mathbb{T} \times \mathbb{R}$.
\item $W^n(\theta) \gets \begin{pmatrix} 0 \\ 0 \end{pmatrix}$.
\item [$s^1$] If $n=1$, $W^1(\theta) \gets \begin{pmatrix} 0
  \\ 1 \end{pmatrix} $ and $\lambda \gets \lambda _0$.
\item\label{alg.sn-loop} $W(\theta,s) \gets
\begin{pmatrix}
\theta \\ 0
\end{pmatrix} + \hat W^0(\theta) + \sum \limits _{j=1}^{n} W^j(\theta) (bs)^j$.
\item\label{alg.sn-linear-system} $Y(W(\theta,s)) \gets DK\circ
  W(\theta,s)^{-1} P(K\circ W(\theta,s), K\circ\widetilde{
    W}(\theta,s), \varepsilon)$.
\item\label{eq.sn-traje} $\eta(\theta)\gets
  \varepsilon\frac{\partial^n Y}{\partial s^n} (W(\theta,s)) _{|s =
    0}$. Let $\eta \equiv (\eta_1, \eta_2)$.
\item [$s^1$] If $n=1$, then $\lambda \gets \lambda _0 + \int _0^1
  \eta_2(\theta)\, d\theta$.
\item Solve $(\omega \partial _\theta + n\lambda) W _1^n(\theta) =
  \eta_1(\theta)$.
\item Solve $(\omega \partial _\theta + n\lambda - \lambda _0)W
  _2^n(\theta) = \eta_2(\theta)$.
\item Iterate~\eqref{alg.sn-loop} until convergence. Then undo the
  scaling $b$.
\end{enumerate}
\end{alg}
Both algorithms~\ref{alg.s0} and~\ref{alg.sn} have non-trivial parts,
such as, the effective computation of $\widetilde W$, the numerical
composition of $K$ with $W$ and also with $\widetilde W$ (see
\S\ref{sec.composing-periodic-maps}), the effective computation of the
step~\ref{eq.sn-traje} in Algorithm~\ref{alg.sn}, the stopping
criterion (see \S\ref{sec.stop-criterion}) and the choice of the
scaling factor (see \S\ref{sec.scaling-factor}).  On the other hand,
there are steps that we can use the same methods in the unperturbed
case, such as, the solution of linear systems like
step~\ref{alg.sn-linear-system} in Algorithm~\ref{alg.sn} via
Lemma~\ref{lem.poly-linear-system} or the solutions of the
cohomological equations via Proposition~\ref{pro.Fourier-solution}.

\subsubsection{Stopping criterion}
\label{sec.stop-criterion}
algorithms~\ref{alg.s0} and~\ref{alg.sn} require to stop the
iterations when prescribed tolerances have been reached. 
Alternatively, one can stop when the invariance equation is satisfied
up to a given tolerance.

\subsubsection{Scaling factor}
\label{sec.scaling-factor}
As in the unperturbed case, if $W(\theta, s)$ is a solution, then
$W(\theta + \theta _0, bs)$ will be a solution too for any $\theta _0$
and $b$. A difference with the $\varepsilon=0$ case is that now
$K\circ W$ and $K\circ \widetilde W$ are required to be well-defined.
That means the second components of $W$ and $\widetilde W$ must lie in
$[-1,1]$.  Stronger conditions are
\[
p(s)= \sum _{j \geq 0} \lVert W _2^j(\theta) \rVert \lvert s \rvert ^j
\leq 1 \qquad \text{and} \qquad \widetilde{p}(s) = \sum _{j \geq 0}
\lVert \widetilde{W _2^j}(\theta) \rVert \lvert s \rvert ^j \leq 1.
\]
In the iterative scheme of Algorithm~\ref{alg.sn}, these series become
finite sums and a condition for the value $b > 0$ is led by the
upper-bound $\min\{s^\ast, \widetilde s^\ast\}$ where $s^\ast$ is the
value so that $p(s ^\ast) = 1$ and, similarly, $\widetilde s^*$ the
value verifying $\widetilde p(\widetilde s ^\ast) = 1$.  Notice that,
the solutions $s^\ast$ and $\widetilde s^\ast$ exist because $\lVert W
_2^0(\theta) \rVert < 1$, $\lVert \widetilde{W _2^0}(\theta) \rVert <
1$ and the polynomials are strictly positive for $s\geq 0$.


\section{Numerical composition of periodic maps}
\label{sec.composing-periodic-maps}
The goal of this section is to deeply discuss how we can numerically
compute $\widetilde W$, the compositions $K\circ W(\theta,s)$ and
$K\circ \widetilde W(\theta, s)$ only having a numerical
representation (or approximation) of $K$ and $W$ in the algorithms
\ref{alg.s0} and \ref{alg.sn}.

There are a variety of methods that can be employed to numerically get
the composition of a periodic mapping with another (or the same)
mapping.  Some of these methods depend strongly on the representation
of the periodic mapping and others only depend on specific parts of
the algorithm.

We start the discussion from the general methods to those that
strongly depend on the numerical representation. One expects that the
general ones will have a bigger numerical complexity or it will be
less accurate.

Before starting to discuss the algorithms, it is important to stress
again that for functions of two variables $(\theta, s) \in \torus
\times [-1,1]$, there are two complementary ways of looking at
them. We can think of them as functions that given $\theta$ produce a
polynomial in $s$ -- this polynomial valued function will be periodic
in $\theta$ -- or we can think of them as polynomials in $s$ taking
values in spaces of periodic functions (of the variable $\theta$).  Of
course, the periodic functions that appear in our interpretation can
be discretized either by the values in a grid of points or by the
Fourier transform.

Each of these -- equivalent! -- interpretations will be useful in some
algorithms.  In the second interpretation, we can ``overload''
algorithms for standard polynomials to work with polynomials whose
coefficients are periodic functions (in particular Horner schemes).
In the first interpretation, we can easily parallelize algorithms for
polynomials for each of the values of $\theta$ using the grid
discretization of periodic functions.

Possibly the hardest part of algorithms~\ref{alg.s0} and~\ref{alg.sn}
is the compositions between $K$ with $W$ and with $\widetilde W$. Due
to the step~\ref{eq.sn-traje} of Algorithm~\ref{alg.sn} the
composition should be done so that the output is still a polynomial in
$s$ with coefficients that are periodic functions of $\theta$.

In our implementation, we use the Automatic Differentiation (AD)
approach~\cite{Haro2016,Griewank2008}.

If $W(\theta,s) = (W _1(\theta,s),W _2(\theta,s))$ is a function of
two variables taking values in $\real^2$, then
\begin{equation}\label{eq.KW}
 K \circ W(\theta, s) = \sum _{j = 0}^{m-1} K ^j (W _1(\theta, s)))
 \left( b _0 W _2(\theta, s) \right)^j,
\end{equation}
which can be evaluated with $m-1$ polynomial products and $m-1$
polynomial sums using Horner scheme, once we have computed $K ^j \circ
W_1(\theta,s)$.

The problem of composing a periodic function with a periodic
polynomial in $s$ -- to produce a polynomial in $s$ taking values in
the space of periodic functions -- is what we consider now.

The most general method considers $\pfun$ a periodic function, the
$K^j$ in \eqref{eq.KW}, and $q(s)=\sum_{j=0}^k q_j s^j$ a polynomial
of a fixed order $k\geq 0$ where the $q_j$ are periodic functions of
$\theta$ that we consider discretized by their values in a grid.


We want to compute the polynomial $p\coloneq \pfun \circ q$ up to
order $k$.  Assume that $\frac{d ^ j}{d \theta^j}\pfun (q _ 0)$ for $0
\leq j \leq k$ are given as input and that they have been previously
computed in a bounded computational cost.  The chain rule gives us a
procedure to compute the coefficients of $p(s)=\sum_{j = 0}^k p _j
s^j$.

Indeed, one can build a table, whose entries are polynomials in $s$,
like in Table~\ref{table.traje-compo} following the generation rule in
Figure~\ref{fig.table-rule}.

\begin{figure}[ht]
 \begin{center}
  \includegraphics[scale=1]{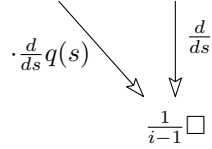}
 \end{center} 
 \caption{Generation rule for $i = 2, \dotsc, k+1$
   Table~\ref{table.traje-compo} entries}
\label{fig.table-rule} 
\end{figure}

The inputs of Table~\ref{table.traje-compo} are $a _{i,1} = 0$ for $i
\ne 1$ and $a _{2, 2} = \frac{d}{ds} q(s)$. Then the entries $a_{ij}$
with $2\leq j \leq i \leq k+1$ are given by
\begin{equation}\label{eq.compo-rule}
 a _{ij}(s) = \frac{1}{i-1} \left( \frac{d}{ds} a _{i-1,j}(s) +
  a _{i-1,j-1} (s)\frac{d}{ds}q(s)\right).
\end{equation}
Thus, the coefficients of $p(s)$ are $p_j=\sum _{l=0}^{k} a _{jl}(0)
\frac{d^l}{d\theta^l} \pfun (q_0)$ for $0 \leq j \leq k$.

\begin{table}[ht] 
 \[
   \renewcommand{\arraystretch}{2}
   \begin{array}{l||*{6}{c}}
     & \pfun (q_0) & \frac{d}{d\theta} \pfun (q_0) &
     \frac{d^2}{d\theta^2} \pfun (q_0) & \cdots &
     \frac{d^{k-1}}{d\theta^{k-1}} \pfun (q_0) &\frac{d^k}{d\theta^k}
     \pfun (q_0) \\ \hline \hline p_0 & 1 & 0 \\ p_1 & 0 &
     \frac{d}{ds} q(s) & 0 \\ p_2 & 0 &\frac{1}{2}\square &
     \frac{1}{2}\square \\ \vdots & \vdots & \vdots & \vdots & \ddots
     & 0 \\ p_{k-1} & 0 & \frac{1}{k-1}\square & \frac{1}{k-1}\square
     & \cdots & \frac{1}{k-1}\square & 0 \\ p_{k} & 0 &
     \frac{1}{k}\square & \frac{1}{k}\square & \cdots &
     \frac{1}{k}\square & \frac{1}{k}\square \\
   \end{array}
 \] 
 \caption{Composition of a function with a polynomial}
 \label{table.traje-compo}
\end{table}

Note that it is enough to store in memory $k$ entries of the
Table~\ref{table.traje-compo} to compute all the coefficients
$p _j$.  

Moreover, for each entry in the $i$th row with $i=2,\dotsc, k+1$, one
only needs to consider polynomials of degree $k+1-i$. Overall the
memory required is at most $\frac{1}{2}k(k+1)$. The number of
arithmetic operations following the rule~\eqref{eq.compo-rule} are
given by the Proposition~\ref{pro.table-complexity}.

\begin{pro}\label{pro.table-complexity}
 Let $\pfun $ be a real-periodic function and let $q(s)$ be a real
 polynomial of degree $k$.  Given $\frac{d^j}{d\theta^j} \pfun (q(0))$
 for $j = 0, \dotsc, k$.  The polynomial $\pfun \circ q$ can be
 performed using Table~\ref{table.traje-compo} with
 $\frac{1}{2}k(k+1)$ units of memory and $\Theta(k^4)$ multiplications
 and additions.
\end{pro}
\begin{proof}
 Note that $k(k+1)$ multiplications and $(k+1)^2$ additions are needed
 to perform the product of two polynomials of degree $k$. Also $k$
 multiplications are needed to perform the derivative of a polynomial
 of degree $k$ multiplied by a scalar.  To bound the number of
 operations we must distinct three different situations of the
 Table~\ref{table.traje-compo}.
 \begin{enumerate}
  \item The column $a _{3..k, 2}$. $\sum \limits _{i=1}^{k-2} (k-i+1)
    = \frac{1}{2}(k^2+k-6)$ multiplications.
  \item The diagonal $a _{3..k, 3..k}$.
  \begin{itemize}
   \item $\sum \limits _{j=1}^{k-2} (k-j-1)(k-j+1)+1 = \frac{1}{6} (2
     k^3-3 k^2+k-6)$ multiplications.
   \item $\sum \limits _{j=1}^{k-2} (k-j-1)^2 + 1 = \frac{1}{6} (2
     k^3-9 k^2+19 k-18)$ additions.
  \end{itemize}
  \item The rest.
  \begin{itemize}
   \item $\sum \limits _{j=1}^{k-2} \sum \limits _{i=j+1}^{k-2}
     (k-i-1)(k-i+1) + (k-i-2) +1 = \frac{1}{12} (7 k^4-56 k^3+71
     k^2+38 k-24)$ multiplications.
   \item $\sum \limits _{j=1}^{k-2} \sum \limits _{i=j+1}^{k-2}
     (k-i-1)^2 + (k-i) + 1 = \frac{1}{12} (5 k^4-36 k^3+85 k^2-102
     k+72)$ additions.
  \end{itemize}
 \end{enumerate}
 Overall $\frac{7}{12}k^4 + \Theta(k^3)$ multiplications and
 $\frac{5}{12} k^4 + \Theta(k^3)$ additions.
\end{proof}

The next Theorem~\ref{thm.KW-tKW-complexity} summarizes the previous
explanations and it provides the complexities to numerically compute
$K\circ W$ in \eqref{eq.KW}. It assumes that $\frac{d^i}{d\theta^i}
\pfun (q _0)$ of Table~\ref{table.traje-compo} are given as input
because their computation strongly depends on the numerical
representation of a periodic mapping.
\begin{thm}\label{thm.KW-tKW-complexity}
 For a fixed $\theta$, the computational complexity to compute the
 compositions of $K(\theta, s) = \sum _{j=0}^{m-1} K ^j(\theta) (b _0
 s)^j$ with $W(\theta, s) = \sum _{j=0}^{k-1} W^j(\theta) (b s)^j$ and
 $\widetilde W(\theta, s)$ using Table~\ref{table.traje-compo} is
 $\Theta(mk^4)$ and space $\Omega(k^2)$ assuming
 $\frac{d^i}{d\theta^i} K ^j(W^0_1(\theta))$ as input for $i=0,\dotsc,
 k-1$.
\end{thm}

\begin{remark}
  \label{rem.KW-tKW-complexity}
  In general, if $n _\theta$ denotes the mesh size of the variable
  $\theta$, we will have $k \leq m \ll n _\theta$. That is, the mesh
  size will be much larger than the degree (in $s$) of $K(\theta,s)$.
  That means that the parallelization in $n _\theta$ will be more
  advantageous.

  Theorem~\ref{thm.KW-tKW-complexity} has an important assumption
  involving $\frac{d^i}{d \theta^i} K ^j (W_1^0(\theta))$ which can
  have a big impact in the complexity of $K\circ
  W(\theta,s)$. However, such an impact strongly depends on the
  numerical representation of $K ^j$ and it will be discussed in the
  Fourier representation case.
\end{remark}
 

 \subsection{Composition in Fourier}
 Theorem~\ref{thm.KW-tKW-complexity} reduces the problem of computing
 $K\circ W(\theta,s)$ in \eqref{eq.KW} to the problem of computing
 composition of a periodic function with another one. Such a
 composition of real Fourier truncated series may require to know the
 values not in the standard equispaced mesh of $\theta$ which hampers
 the use of the FFT.  A direct composition of real Fourier series
 requires a computational complexity $\Theta(n _\theta ^ 2)$. However
 it can be performed with $\Theta(n _\theta \log n _\theta)$ by the
 \textsc{nfft3}, see \cite{Keiner2009}. The package \textsc{nfft3}
 allows to express $\pfun \colon \mathbb{T} \rightarrow \mathbb{R}$
 with the same coefficients in \eqref{eq.dft-top} and perform its
 evaluation in an even number of non-equispaced nodes $(\theta _k) _{k
   = 0}^{n _\theta-1}\subset \mathbb{T}$ by
\begin{equation}\label{eq.ndft-top}
 \pfun(\theta _k) = \sum _{j = 0}^{n _\theta - 1} \widehat \pfun ^j
 e^{-2\pi \I (j - n _\theta/2) (\theta _k - 1/2)}.
\end{equation}
The corrections of $\theta _k$ in \eqref{eq.ndft-top} is because
\textsc{nfft3} considers $\mathbb{T} \simeq [-1/2, 1/2)$ rather than
  the other standard equispaced discretization in $[0,1)$.
    \textsc{nfft3} uses some window functions for a first
    approximation as a cut-off in the frequency domain and also for a
    second approximation as a cut-off in time domain. It takes under
    control (by bounds) these approximations to ensure the solution is
    a good approximation.  Joining these result with
    Proposition~\ref{pro.table-complexity} we can rewrite
    Theorem~\ref{thm.KW-tKW-complexity} as
\begin{thm}
\label{thm.KW-tKW-fourier-complexity}
 The computational complexity to compute in Algorithm~\ref{alg.sn} the
 compositions of $K(\theta, s) = \sum _{j=0}^{m-1} K ^j(\theta) (b _0
 s)^j$ with $W(\theta, s) = \sum _{j=0}^{k-1} W^j(\theta) (b s)^j$ and
 $\widetilde W(\theta, s)= \sum _{j=0}^{k-1} \widetilde W^j(\theta) (b
 s)^j$ using Table~\ref{table.traje-compo} and \textsc{nfft3}, and
 assuming that $K ^j$, $W ^j$ and $\widetilde W^j$ are expressed with
 $n _\theta$ Fourier coefficients is $\Theta(m k^4 n _\theta + m k n
 _\theta \log n _\theta)$.  The space complexity is $\Omega(k n
 _\theta + k^2)$.
\end{thm}
\begin{remark}
  The remark~\ref{rem.KW-tKW-complexity} also applies to
  Theorem~\ref{thm.KW-tKW-fourier-complexity} in terms of the
  parallelization of $n _\theta$ due to the fact that in general $k
  \leq m \ll n _\theta$. However, in the parallelism case, the space
  complexity increase to $\Omega(k n _\theta + k^2 n_p)$ with $n _p$
  the number of processes although the part corresponding to $k n
  _\theta$ can be shared memory.

  In particular, the \textsc{nfft3} can also be used for the zero
  order $W^0$ of Algorithm~\ref{alg.s0} giving in that case the same
  complexity as Theorem~\ref{thm.KW-tKW-fourier-complexity} but with
  $k = 1$.
\end{remark}

\subsection{Automatic Differentiation in Fourier}
 \label{sec.ad-fourier}
Theorem~\ref{thm.KW-tKW-complexity} tells us that the composition $K
\circ W (\theta, s)$ can numerically be done independently of the
periodic mapping representation. Nevertheless, differentiation is a
notoriously ill-posed problem due to the lack of information in the
discretized problem. Thus, Table~\ref{table.traje-compo} is a good
option when no advantage of the computer periodic representation
exists or $k \ll m$.

Using the representation \eqref{eq.dft-top-real}, we can use the
Taylor expansion of the sine and cosine by recurrence
\cite{Knuth1981,Haro2016}.  That is, if $q(s)$ is a polynomial, then
$\sin q(s)$ and $\cos q(s)$ are given by $s _0 = \sin q _0$, $c _0 =
\cos q _0$ and for $j \geq 1$,
\begin{equation}
 s _j = \frac{1}{j} \sum _{k=0}^{j-1} (j-k) q _{j-k} c _k, \qquad 
 c _j = -\frac{1}{j} \sum _{k=0}^{j-1} (j-k) q _{j-k} s _k.
\end{equation}
Therefore the computational cost to obtain the sine and cosine of a
polynomial is linear with respect to its degree.

Theorem~\ref{thm.KW-tKW-adfourier-complexity} says that the
composition of $K$ with $W$ or $\widetilde W$ are rather than
$\Theta(m k^4 n _\theta + m k n _\theta \log n _\theta)$ like in
Theorem~\ref{thm.KW-tKW-fourier-complexity} just $\Theta(m k n
_\theta^2)$.  Therefore if $k \ll m$ and $n _\theta$ is large, the
approach given by Theorem~\ref{thm.KW-tKW-fourier-complexity} has a
better complexity although
Theorem~\ref{thm.KW-tKW-adfourier-complexity} will be more stable for
larger $k$.
\begin{thm}
\label{thm.KW-tKW-adfourier-complexity}
 The computational complexity to compute in Algorithm~\ref{alg.sn} the
 compositions of $K(\theta, s) = \sum _{j=0}^{m-1} K ^j(\theta) (b _0
 s)^j$ with $W(\theta, s) = \sum _{j=0}^{k-1} W^j(\theta) (b s)^j$ and
 $\widetilde W(\theta, s)= \sum _{j=0}^{k-1} \widetilde W^j(\theta) (b
 s)^j$ using Automatic Differentiation and assuming that $K ^j$, $W
 ^j$ and $\widetilde W^j$ are expressed with $n _\theta$ Fourier
 coefficients is $\Theta(m k n _\theta ^2)$.
\end{thm}
 
%
%
%
 
 \section{Numerical results}
 \label{sec.experiment}
 The van der Pol oscillator \cite{Vanderpol1920} is an oscillator with
 a non-linear damping governed by a second-order differential
 equation.
 
The state-dependent perturbation of the van der Pol oscillator in
\cite{Hou2015} has the form
 \begin{equation}\label{eq.vdp-delayed}
  \begin{split}
   \dot x(t)&= y(t), \\
   \dot y(t)&= \mu(1-x(t)^2)y(t) - x(t) + \varepsilon x(t-r(x(t))),
  \end{split}
 \end{equation}
with $\mu > 0$ and $0 < \varepsilon \ll 1$. For the delay function
$r(x)$ we are going to consider two cases. A pure state-dependent
delay case $r(x) = 0.006 e^{2 x}$ or just a constant delay case $r(x)
= 0.006$.

The first step consists in computing the change of coordinate $K$, the
frequency $\omega _0$ of the limit cycle and its stability value
$\lambda _0 < 0$ for $\varepsilon = 0$. By standard methods of
computing periodic orbits and their first variational equations, we
compute the limit cycles close to $(x,y)=(2,0)$ for different values
of $\mu$. Table~\ref{table.vdp-omega0-lambda0} shows the values of
$\omega _0$ and $\lambda _0$ for each of those values of the parameter
$\mu$.

\begin{table}[ht]
 \[
  \begin{array}{l|c|c}
   \mu  & \omega _0                   & \lambda _0 \\ \hline \hline
   0.25 & \mathtt{0.1585366857025485} & \mathtt{-0.2509741760777654} \\
   0.5  & \mathtt{0.1567232109993800} & \mathtt{-0.5077310891698608} \\
   1    & \mathtt{0.1500760842377394} & \mathtt{-1.0593769948418550} \\
   1.5  & \mathtt{0.1409170454968141} & \mathtt{-1.6837946490433340} 
  \end{array}
 \]
 \caption{Values of $\omega _0$ and $\lambda _0$ for different values
   of the parameter $\mu$ in eq.~\eqref{eq.vdp-delayed} with
   $\varepsilon = 0$}\label{table.vdp-omega0-lambda0}
\end{table}

The computation of $K(\theta,s)$, following
Algorithm~\ref{alg.huguet}, up to order $16$ in $s$ and with a Fourier
mesh size of $1024$ allows to plot the isochrons in
Figure~\ref{fig.vdp-iso-huguet}.

In the case of ODE's, the isochrons computed by evaluating the
expansion can be globalized by integration of the ODE
\eqref{eq.vdp-delayed} forward and backward in time, see
\cite{Huguet2013}.  In the case of the SDDE, $\varepsilon \ne 0$,
propagating backwards is not possible. We hope that this limitation
can be overcome, but this will require some new rigorous developments
and more algorithms.  We think that this is a very interesting
problem.

 \begin{figure}[ht]
  \begin{center}
    \input{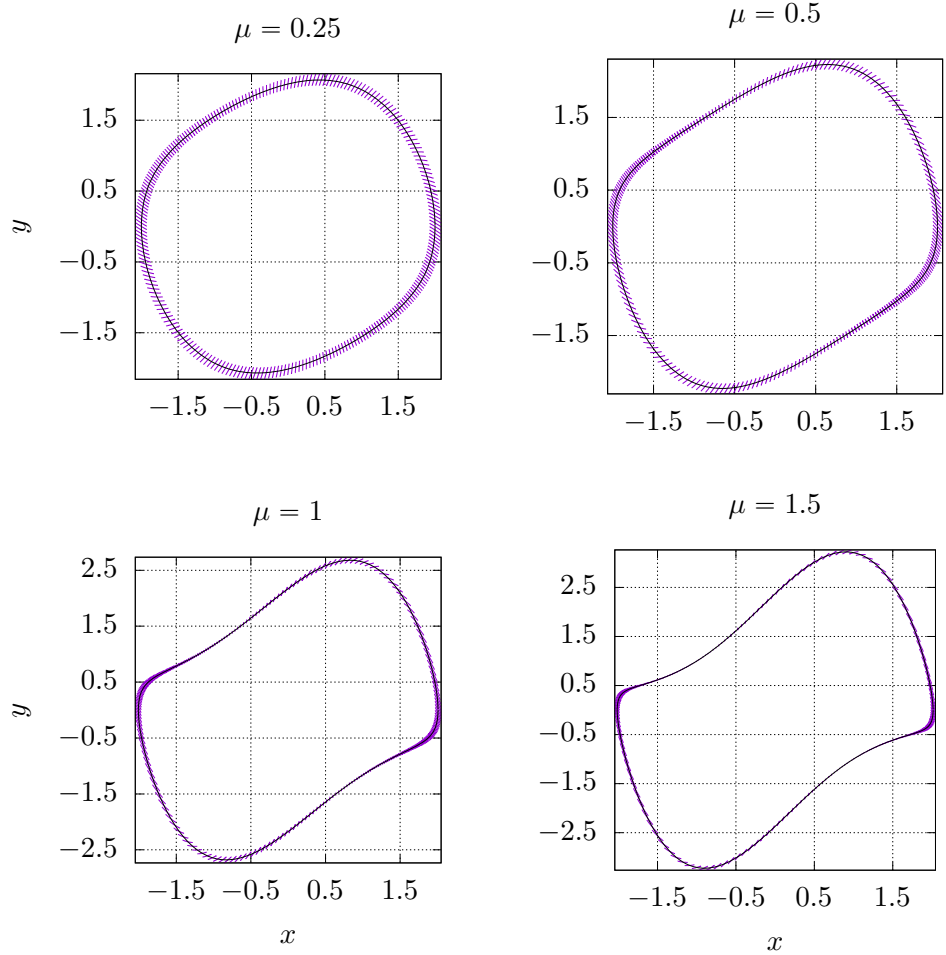}
    \caption{Limit cycles and their isochrons for different values of
      the parameter $\mu$ in the unperturbed,
      eq.~\eqref{eq.vdp-delayed}}
\label{fig.vdp-iso-huguet}
  \end{center}
 \end{figure}
A relevant indicator for engineers is the power spectrum, i.e. the
square of the modulus of the complex Fourier coefficients. In
Figure~\ref{fig.vp-power-huguet} we illustrate the power spectrum for
$K^ 0$, since $K ^0$ is the one that is commonly observed in a circuit
system.
 
 \begin{figure}[ht]
  \begin{center}
    \input{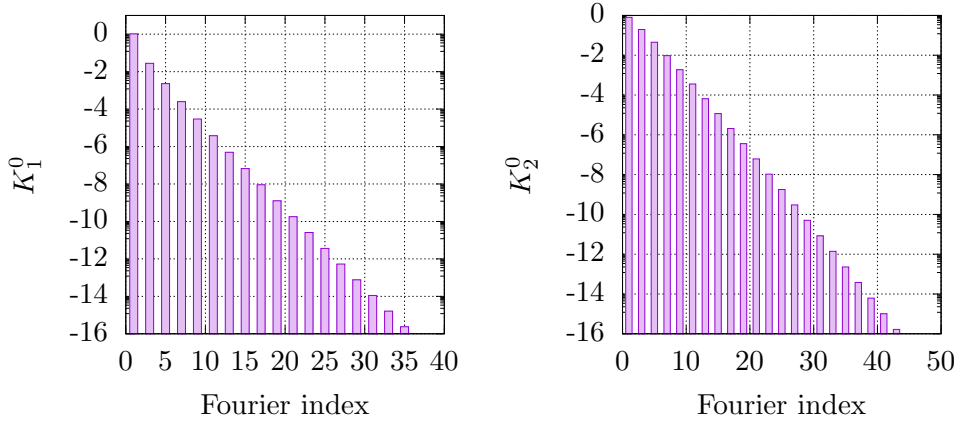}
    \caption{Logscale of the power spectrum of $K^0 \equiv (K^0 _1,
      K^0 _2)$ for $\mu=1.5$ and $\varepsilon = 0$ in
      eq.~\eqref{eq.vdp-delayed}}
\label{fig.vp-power-huguet}
  \end{center}
 \end{figure}

Due to the quadratic convergence of the Algorithm~\ref{alg.huguet},
see \cite{Huguet2013}, the computation of
Table~\ref{table.vdp-omega0-lambda0} and
Figure~\ref{fig.vdp-iso-huguet} are performed in less than one min in
a today standard laptop. However, we notice that for values of $\mu >
1.5$ the method may not converge for the unperturbed case, the scaling
factor and the Fourier mesh size need to be smaller due to spikes,
especially for the high orders in $s$, i.e. $K ^j(\theta)$ for large
$j$. This is an inherent drawback of the numerical representation of
periodic functions that can be emphasized with the model involved.
 
 \subsection{Perturbed case}
Let us analyze the case of $\mu = 1.5$ for two different types of
delay functions; a constant one $r(x) = 0.006$ and a state-dependent
one $r(x) = 0.006 e ^x$.
 
The two cases have some advantages to be exploited. For instance, in
the constant case $\widetilde W(\theta,s) = W(\theta - \omega \beta, s
e^{-\lambda \beta})$ is easier to compute than in the state-dependent
case. Since in both cases $W$ and $\widetilde W$ must be composed by
$K$, the use of automatic differentiation for the step
\ref{eq.sn-traje} in Algorithm~\ref{alg.sn} is still needed. In
particular, for the Algorithm~\ref{alg.s0} and the composition via
Theorem~\ref{thm.KW-tKW-fourier-complexity}, the \textsc{nfft3} can be
used to perform the numerical composition of $K$ with $W$ and
$\widetilde W$.

The first steps of our method get $\omega $ and $\lambda$ which we
distinguish their values depending on the delay function and the
parameter $\varepsilon$. Again here we are assuming $\mu=1.5$. These
values are summarized respectively in Tables~\ref{table.vdp-omega} and
\ref{table.vdp-lambda}. They were computed fixing a tolerance for the
stopping criterion of $10^{-10}$ in double-precision. As one expects
they are close to those in Table~\ref{table.vdp-omega0-lambda0} and
are further as $\varepsilon$ increase. Moreover we report a speed
factor around $2.25$ using the \textsc{nfft3} with respect to a direct
implementation of the Fourier composition.

\begin{table}[ht]
 \[
  \begin{array}{c||c|c}
   \varepsilon  & \omega _s & \omega _c  \\ \hline \hline
   10^{-4} & \mathtt{0.140908673246532} & \mathtt{0.140908547470887} \\
   10^{-3} & \mathtt{0.140833302396846} & \mathtt{0.140832045466042} \\
   10^{-2} & \mathtt{0.140077545298062} & \mathtt{0.140065058638519} 
  \end{array}
 \]
 \caption{Values of $\omega$ for different values of $\varepsilon$ in
   eq.~\eqref{eq.vdp-delayed} with $\mu = 1.5$ obtained by
   Algorithm~\ref{alg.s0}. $\omega _s$ corresponds to the
   state-dependent delay and $\omega _c$ the constant
   delay}\label{table.vdp-omega}
\end{table}

\begin{table}[ht]
 \[
  \begin{array}{c||c|c}
   \varepsilon  &  \lambda _s &  \lambda _c  \\ \hline \hline
   10^{-4} & \mathtt{-1.6838123845562083} 
           & \mathtt{-1.6838091880373793} \\
   10^{-3} & \mathtt{-1.6839721186835845} 
           & \mathtt{-1.6839401491442914} \\
   10^{-2} & \mathtt{-1.6855808865357260} 
           & \mathtt{-1.6852607528946115} 
  \end{array}
 \]
 \caption{Values of $\lambda$ for different values of $\varepsilon$ in
   eq.~\eqref{eq.vdp-delayed} with $\mu = 1.5$ obtained by
   Algorithm~\ref{alg.sn}. $\lambda _s$ corresponds to the
   state-dependent delay and $\lambda _c$ the constant
   delay}\label{table.vdp-lambda}
\end{table}

Figure~\ref{fig.vp-inv} shows, for different values of $\varepsilon$
in eq.~\eqref{eq.vdp-delayed}, the logarithmic error of invariance
equation for each of the different orders $j \geq 0$. That is, the
finite system of invariance equations obtained after plugging $W
(\theta, s) = \sum W ^j(\theta) s^j$ into eq.~\eqref{eq.inv} and
matching terms of the same order. The state-dependent case needs
smaller values of $\varepsilon$ to satisfy the invariance equation
while the constant delay case admits larger values of $\varepsilon$
which can be deduced from the inequalities in \cite{Jiaqi2020}.

 \begin{figure}[ht]
  \centering
  \input{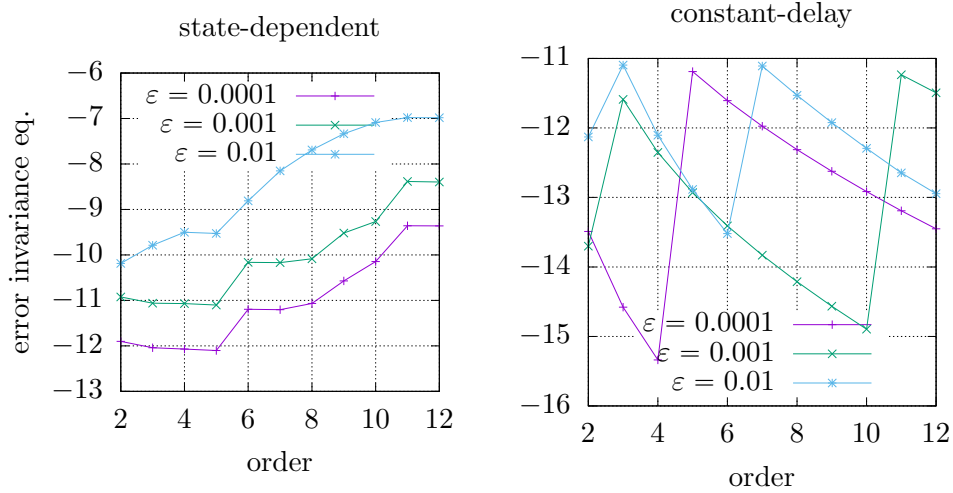}
  \caption{Log10 scale of the 2-norm of the error in the invariance
    equation}\label{fig.vp-inv}
 \end{figure}

 Figures~\ref{fig.vp-norm} shows the difference between the isochrons
 for the perturbed and unperturbed case. As one expects from the
 theorems in \cite{Jiaqi2020}, the error is smaller as the
 perturbation parameter value $\varepsilon$ becomes smaller.
 
 \begin{figure}[ht]
  \centering
  \input{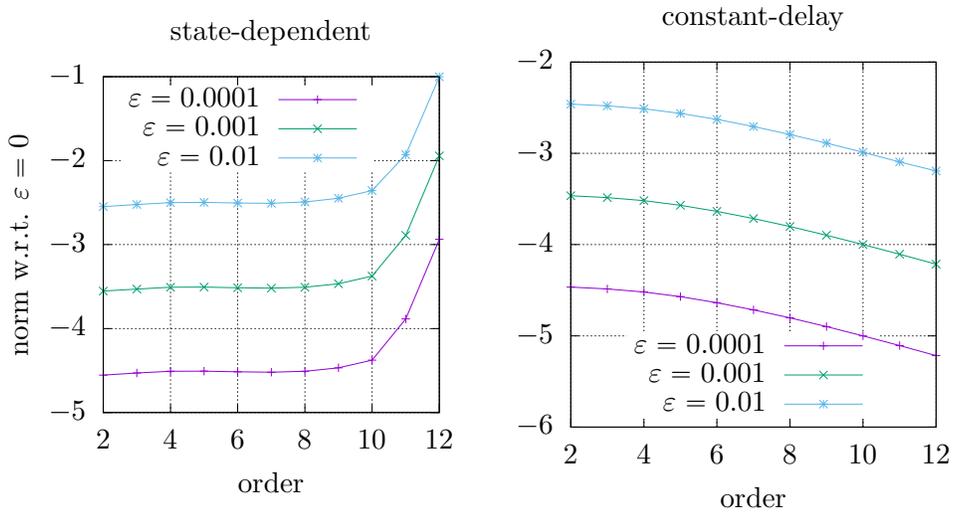}
  \caption{Log10 scale of the 2-norm of the difference between the
    perturbed and unperturbed cases. That is, $\lVert K ^ j - (K\circ
    W)^j \rVert$}\label{fig.vp-norm}
 \end{figure}
 
 An important point in Algorithm~\ref{alg.sn} is the well-definedness
 of the composition of $K$ with $W$ and $\widetilde W$. Because the
 state-dependent delays consider much more situations than just the
 constant delay, one expects that potentially smaller scaling factor
 compared to the constant delay will be needed as large order is
 computed. Figure~\ref{fig.vp-scaling} shows if $\varepsilon$ is far
 from the unperturbed case it will need to be smaller, that for the
 constant case is enough to use a constant scaling factor and for the
 state-dependent it decrease drastically in the first orders.
 
 \begin{figure}[ht]
  \centering
  \input{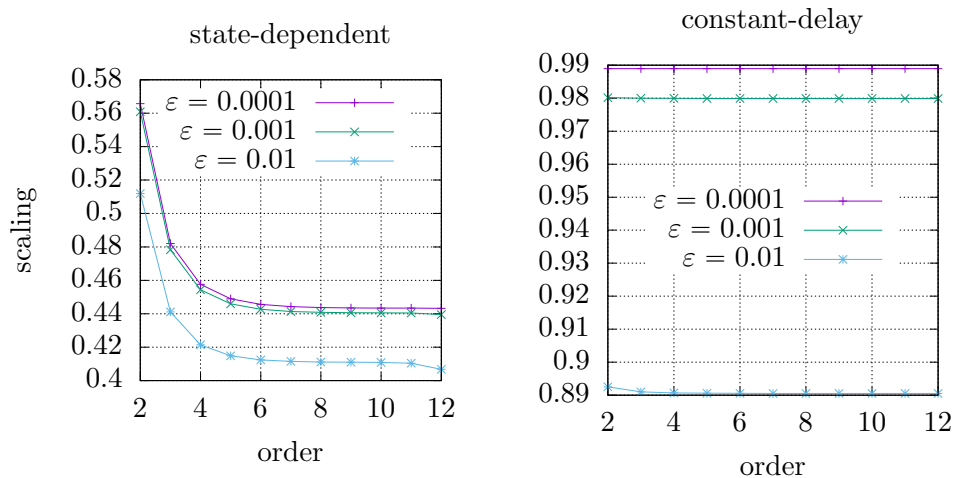}
  \caption{Scaling factor to ensure that the composition of $K$ with
    $W$ and with $\widetilde W$ in Algorithm~\ref{alg.sn} are
    well-defined}\label{fig.vp-scaling}
 \end{figure}
 
 To illustrate the physical observation the Figures~\ref{fig.vp-power}
 and \ref{fig.vp-ct-power} shows the power spectra of the limit cycles
 after the perturbations. More concretely, Figure~\ref{fig.vp-power}
 displays the power spectrum of $(K\circ W)^0$ for the pure
 state-dependent delay case and $\varepsilon = 0.01$. In contrast with
 Figure~\ref{fig.vp-power-huguet}, we observe that for the even
 indexes they have non-zero values in the double-precision arithmetic
 sense.  On the other hand, Figure~\ref{fig.vp-ct-power} shows that
 these non-zero values in the even indexes are not present in the
 constant delay case and the power spectrum for the case $\varepsilon
 > 0$ is away from that when $\varepsilon=0$ as $\varepsilon$
 increase.
 
 \begin{figure}[ht]
  \centering
  \input{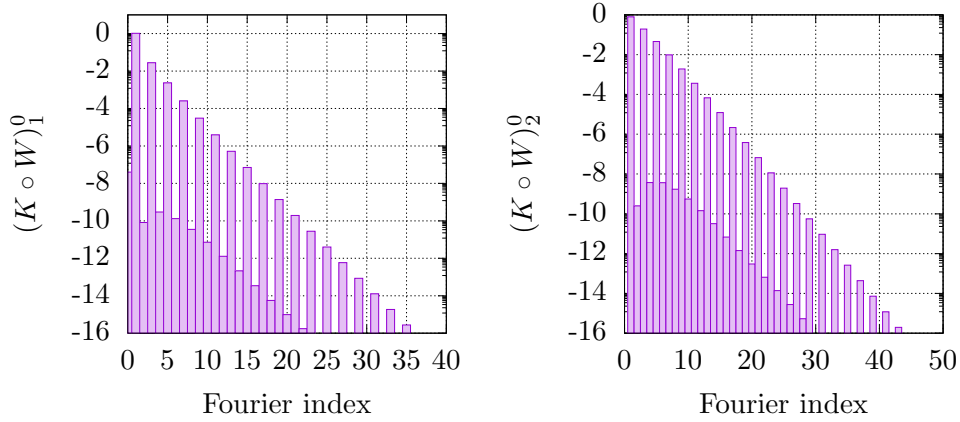}
  \caption{Log10 scale of the power spectrum of $(K\circ W)^0$ for
    $\mu=1.5$, $\varepsilon = 0.01$ and the state-dependent delay
    $r(x) = 0.006e^x$ in
    eq.~\eqref{eq.vdp-delayed}}\label{fig.vp-power}
 \end{figure}
 
 \begin{figure}[ht]
  \centering
  \input{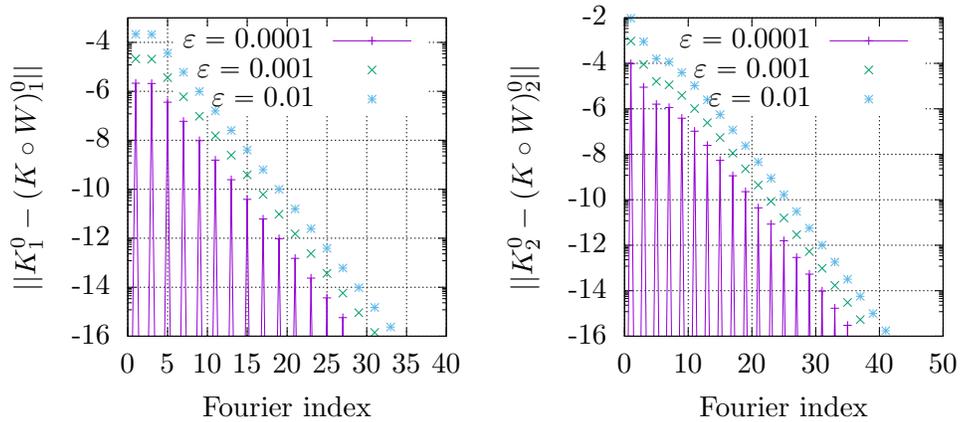}
  \caption{Log10 scale of the difference between the power spectrum of
    $K^0$ and the power spectrum of $(K\circ W)^0$ for $\mu = 1.5$,
    different values of $\varepsilon$ and constant delay $r = 0.006$
    in eq.~\eqref{eq.vdp-delayed}}\label{fig.vp-ct-power}
 \end{figure}

 \bibliographystyle{alpha}
 \bibliography{ref.bib}
\end{document}